\documentclass[english]{article}
\usepackage[T1]{fontenc}
\usepackage[latin9]{inputenc}
\usepackage{color}
\usepackage{array}
\usepackage{float}
\usepackage{multirow}
\usepackage{amsmath}
\usepackage{amsthm}
\usepackage{amssymb}
\usepackage{graphicx}

\makeatletter

\providecommand{\tabularnewline}{\\}
\floatstyle{ruled}
\newfloat{algorithm}{tbp}{loa}
\providecommand{\algorithmname}{Algorithm}
\floatname{algorithm}{\protect\algorithmname}

\numberwithin{equation}{section}
\numberwithin{figure}{section}
\newcommand{\lyxaddress}[1]{
	\par {\raggedright #1
	\vspace{1.4em}
	\noindent\par}
}
\theoremstyle{plain}
\newtheorem{thm}{\protect\theoremname}
\theoremstyle{definition}
\newtheorem{defn}[thm]{\protect\definitionname}
\theoremstyle{plain}
\newtheorem{lem}[thm]{\protect\lemmaname}
\newenvironment{lyxlist}[1]
	{\begin{list}{}
		{\settowidth{\labelwidth}{#1}
		 \setlength{\leftmargin}{\labelwidth}
		 \addtolength{\leftmargin}{\labelsep}
		 }}
	{\end{list}}
\theoremstyle{remark}
\newtheorem{rem}[thm]{\protect\remarkname}
\theoremstyle{plain}
\newtheorem{prop}[thm]{\protect\propositionname}
\theoremstyle{remark}
\newtheorem*{acknowledgement*}{\protect\acknowledgementname}

\usepackage{algorithmic}

\makeatother

\usepackage{babel}
\providecommand{\acknowledgementname}{Acknowledgement}
\providecommand{\definitionname}{Definition}
\providecommand{\lemmaname}{Lemma}
\providecommand{\propositionname}{Proposition}
\providecommand{\remarkname}{Remark}
\providecommand{\theoremname}{Theorem}

\begin{document}
\title{A numerical analysis of planar central and balanced configurations
in the $(n+1)$-body problem with a small mass}
\author{Alexandru Doicu$^{1}$\thanks{E-mail address: alex.doicua@gmail.com},
Lei Zhao$^{2}$\thanks{E-mail address: lei.zhao@math.uni-augsburg.de}
and Adrian Doicu$^{3}$\thanks{Corresponding author. E-mail address: adrian.doicu@dlr.de}}
\maketitle

\lyxaddress{$^{1}$former: Institut für Mathematik, Universität Ausgburg, Augsburg
86135, Germany}

\lyxaddress{$^{2}$Institut für Mathematik, Universität Ausgburg, Augsburg 86135,
Germany}

\lyxaddress{$^{3}$Institut für Methodik der Fernerkundung (IMF), Deutsches Zentrum
für Luft- und Raumfahrt (DLR), Oberpfaffenhofen 82234, Germany }
\begin{abstract}
Two numerical algorithms for analyzing planar central and balanced
configurations in the $(n+1)$-body problem with a small mass are
presented. The first one relies on a direct solution method of the
$(n+1)$-body problem by using a stochastic optimization approach,
while the second one relies on an analytic-continuation method, which
involves the solutions of the $n$-body and the restricted $(n+1)$-body
problem, and the application of a local search procedure to compute
the final $(n+1)$-body configuration in the neighborhood of the configuration
obtained at the first two steps. Some exemplary central and balanced
configurations in the cases $n=4,5,6$ are shown.
\end{abstract}

\section{Introduction}

The $n$-body problem is the problem of predicting motions of a group
of celestial objects interacting with each other gravitationally.
A central configuration is an initial configuration such that if the
particles were all released with zero velocity, they would all collapse
toward the center of mass at the same time. The central configurations
are important in $n$-body problems because they are (i) bifurcation
points for the topological classification of the coplanar $n$-body
problem, and (ii) the starting points for finding some new classes
of periodic solutions. 

A central configuration in the $(n+1)$-body problem is a configuration
which is the limit of central configurations in the full $(n$+1)-body
problem as the mass of the $(n+1)$-th particle tends to zero while
the remaining particles approach definite positive values (Hagihara
1970). Thus central configuration in the $(n+1)$-body problem must
be continuable to positive masses, and this is achieved by a proper
non-degeneracy condition. The restricted three-body problem (both
circular and elliptical) was explored extensively during the 19th
and 20th century. The problem has undergone extensive treatment both
in terms of analytical and numerical tools. Inspired by the restricted
three-body problem, the study of the four-body problem was also simplified
by firstly considering one of the bodies to be of negligible mass
and position of the other three bodies according to different configurations
such as equilateral triangular configuration or bicircular formation.
In this context, it should be mentioned that Arenstorf (1982) obtained
the number of central configurations in the four-body problem by starting
with one zero mass and then analytically continue it into positive
masses. By also using the method of analytic continuation, Xia (1991)
found the exact numbers of central configurations for some open sets
of $n$ positive masses for any choice of $n$. Starting with two
zero masses, the corresponding central configurations, especially
the ones in which two zero masses are at the same point, were obtained.
Under certain conditions, these central configurations were analytically
continued into a full $n$-body problem with all masses positive.

In this work we analyze planar central and balanced configurations
in the $(n+1)$-body problem from a numerical point of view. The case
considered here consists in the computation of central and balanced
configurations when the masses $m_{1},\ldots,m_{n},m_{n+1}$ are given,
and $m_{n+1}$ is finite but very small as compared to $m_{i}$, $i=1,\ldots,n$.
For this purpose, we will use (i) a stochastic optimization approach
(Doicu et al. 2020) for solving directly the $(n+1)$-body problem,
and (ii) a new algorithm relying on an analytic-continuation method. 

The paper is organized as follows. A succinct mathematical description
of central and balanced configurations in the $n$-body problem, as
well as an overview of the stochastic optimization algorithm are provided
in Section 2. In Section 3, the analytic-continuation method for computing
central and balanced configurations in the $(n+1)$-body problem with
a small mass is discussed, and the underlying algorithm is described.
Numerical results are given in Section 4, and some conclusions are
summarized in Section 5.

\section{Planar central and balanced configurations in the $n$-body problem}

Consider $n$ point masses $m_{1},\ldots,m_{n}>0$ with positions
$\mathbf{q}_{1},\ldots,\mathbf{q}_{n}$, where $\mathbf{q}_{i}=(x_{i},y_{i})^{T}\in\mathbb{R}^{2}$.
Define the mass and configuration vectors $\mathbf{m}=(m_{1},\ldots,m_{n})^{T}\in\mathbb{R}^{n+1}$
and $\mathbf{q}=(\mathbf{q}_{1}^{T},\ldots,\mathbf{q}_{n}^{T})^{T}\in\mathbb{R}^{2n}$,
respectively, and let 
\begin{align*}
\Delta & =\{\mathbf{q}=(\mathbf{q}_{1}^{T},\ldots,\mathbf{q}_{n}^{T})^{T}\in\mathbb{R}^{2n}\mid\mathbf{q}_{i}=\mathbf{q}_{j}\text{ for some }i\not=j\},
\end{align*}
be the subspace of $\mathbb{R}^{2n}$ consisting of collisions, 
\begin{equation}
U_{n}(\mathbf{m},\mathbf{q})=\sum_{1\leq i<j\leq n}\frac{m_{i}m_{j}}{||\mathbf{q}_{j}-\mathbf{q}_{i}||},\label{eq:NewtonPot}
\end{equation}
the Newtonian force function for the configuration $\mathbf{q}\in\mathbb{R}^{2n}\backslash\Delta$,
where $||\cdot||$ is the Euclidean norm in $\mathbb{R}^{2}$, 
\begin{equation}
\nabla_{i}U_{n}(\mathbf{m},\mathbf{q})=\left(\begin{array}{c}
\dfrac{\partial U_{n}}{\partial x_{i}}(\mathbf{m},\mathbf{q})\\
\dfrac{\partial U_{n}}{\partial y_{i}}(\mathbf{m},\mathbf{q})
\end{array}\right)=\sum_{\substack{j=1\\
j\not=i
}
}^{n}\frac{m_{i}m_{j}}{||\mathbf{q}_{j}-\mathbf{q}_{i}||^{3}}(\mathbf{q}_{j}-\mathbf{q}_{i})\label{eq:DU}
\end{equation}
the gradient of $U_{n}$ with respect to the coordinates of $\mathbf{q}_{i}$,

\begin{equation}
\mathbf{c}(\mathbf{m},\mathbf{q})=\Bigl(\sum_{i=1}^{n}m_{i}\Bigr)^{-1}\sum_{i=1}^{n}m_{i}\mathbf{q}_{i}\in\mathbb{R}^{2}\label{eq:CM}
\end{equation}
the center of mass of the system of point masses, and $\mathbf{S}\in\mathbb{R}^{2\times2}$
a positive definite symmetric matrix. 
\begin{defn}
\label{def:A-configuration-}A configuration $\mathbf{q}=(\mathbf{q}_{1}^{T},\ldots,\mathbf{q}_{n}^{T})^{T}\in\mathbb{R}^{2n}\backslash\Delta$
is said to form a balanced configuration with respect to the matrix
$\mathbf{S}$ (in short $\text{BC}(\mathsf{\mathbf{S}})$) if there
exists a $\lambda\in\mathbb{R}\backslash\{0\}$ such that the equations
\begin{equation}
\nabla_{i}U_{n}(\mathbf{m},\mathbf{q})+m_{i}\lambda\mathbf{S}(\mathbf{q}_{i}-\mathbf{c}(\mathbf{m},\mathbf{q}))=\mathbf{0},\label{eq:1}
\end{equation}
are satisfied for all $i=1,\ldots,n$. A configuration $\mathbf{q}=(\mathbf{q}_{1}^{T},\ldots,\mathbf{q}_{n}^{T})^{T}\in\mathbb{R}^{2n}\backslash\Delta$
is said to form a central configuration (in short CC) if there exists
a $\lambda\in\mathbb{R}\backslash\{0\}$ for which Eqs. (\ref{eq:1})
are satisfied with $\mathbf{S}=\sigma\mathbf{I}_{2\times2}$ for some\textcolor{red}{{}
}$\sigma\in\mathbb{R}_{+}$. 
\end{defn}

Obviously, central configurations are special cases of balanced configurations
with e.g. $\mathbf{S}=\mathbf{I}_{2\times2}$. 

Consider the diagonal action of $\text{O}(2)$ on $\mathbb{R}^{2n}$,
defined by
\begin{align}
\text{O}(2)\times\mathbb{R}^{2n}\backslash\Delta & \rightarrow\mathbb{R}^{2n}\backslash\Delta\nonumber \\
(\mathbf{O},\mathbf{q}) & \mapsto\mathbf{O}\mathbf{q},\label{eq:1a}
\end{align}
where 
\[
\mathbf{O}\mathbf{q}=\left(\begin{array}{c}
\mathbf{O}\mathbf{q}_{1}\\
\vdots\\
\mathbf{O}\mathbf{q}_{n}
\end{array}\right),
\]
We note the following result, which is a direct consequence of Definition
\ref{def:A-configuration-}.
\begin{lem}
\label{lem:Let--be}Let $\mathbf{q}\in\mathbb{R}^{2n}\backslash\Delta$
be a $\text{BC}(\mathbf{S})$ and $\mathbf{O}\in\text{O}(2)$ an orthogonal
matrix. Then $\mathbf{O}\mathbf{q}$ is a $\text{BC}(\mathbf{O}\mathbf{S}\mathbf{O}^{T})$.
\end{lem}

Two direct consequences of this lemma are the following results:
\begin{enumerate}
\item If $\mathbf{q}\in\mathbb{R}^{2n}\backslash\Delta$ forms a central
configuration and $\mathbf{O}\in\text{O}(2)$ is an orthogonal matrix,
then $\mathbf{O}\mathbf{q}$ is also a central configuration with
the same $\lambda$.
\item The positive definite $2\times2$ matrix $\mathbf{S}$ can be assumed
to be diagonal, i.e.,
\begin{equation}
\mathbf{S}=\left(\begin{array}{cc}
\sigma_{x} & 0\\
0 & \sigma_{y}
\end{array}\right)\label{eq:6}
\end{equation}
with $\sigma_{x},\sigma_{y}>0$ (indeed, if $\mathbf{q}$ is a $\text{BC}(\mathbf{S})$
and $\mathbf{O}\in\text{O}(2)$ is an orthogonal matrix such that
$\mathbf{O}\mathbf{S}\mathbf{O}^{T}=\text{diag}(\sigma_{x},\sigma_{y})$,
then $\mathbf{O}\mathbf{q}$ is a $\text{BC}(\mathbf{O}\mathbf{S}\mathbf{O}^{T})=\text{BC}(\text{diag}(\sigma_{x},\sigma_{y}))$). 
\end{enumerate}
In view of these results, we assume in the following that the matrix
$\mathbf{S}$ is diagonal, i.e., $\mathbf{S}$ is as in Eq. (\ref{eq:6})
with $\sigma_{x},\sigma_{y}>0$.

For $\boldsymbol{\xi},\boldsymbol{\eta}\in\mathbb{R}^{2}$, we define
their inner product with respect to the positive definite diagonal
matrix $\mathbf{S}$ by 
\begin{equation}
\bigl\langle\boldsymbol{\xi},\boldsymbol{\eta}\bigr\rangle_{\mathbf{S}}:=\boldsymbol{\xi}^{T}\mathbf{S}\boldsymbol{\eta},\quad||\boldsymbol{\xi}||_{\mathbf{S}}^{2}=\boldsymbol{\xi}^{T}\mathbf{S}\mathbf{\boldsymbol{\xi}},\label{eq:ScalProd}
\end{equation}
and accordingly, the \emph{$\mathbf{S}$-}weighted moment of inertia
by 
\begin{equation}
I_{\mathbf{S}}(\mathbf{m},\mathbf{q})=\sum_{j=1}^{n}m_{j}||\mathbf{q}_{j}-\mathbf{c}||_{\mathbf{S}}^{2}.\label{eq:Is}
\end{equation}
In this context, assuming that the configuration $\mathbf{q}$ forms
a $\text{BC}(\mathbf{S})$, taking the inner product of Eq. (\ref{eq:1})
with $\mathbf{q}_{i}-\mathbf{c}$, and summing up over all $i=1,\ldots,n$,
we find
\begin{equation}
\lambda=\frac{U(\mathbf{m},\mathbf{q})}{I_{\mathbf{S}}(\mathbf{m},\mathbf{q})}>0.\label{eq:Lambda}
\end{equation}
Thus, in the definition of balanced configurations, the parameter
$\lambda$ cannot be chosen arbitrary; it depends on $\mathbf{q}$
and $\mathbf{S}$. 

In the next step, we consider the \emph{$\mathbf{S}$}-normalized
configuration space defined by
\[
\mathcal{N}(\mathbf{m},\mathbf{S})=\{\mathbf{q}\in\mathbb{R}^{2n}\backslash\Delta\mid\mathbf{c}(\mathbf{m},\mathbf{q})=0,I_{\mathbf{S}}(\mathbf{m},\mathbf{q})=1\}\subset\mathbb{R}^{2n}.
\]
Starting from a BC with respect to $\mathbf{S}$, it is possible to
normalize this configuration so that the new configuration is a BC
with respect to $\mathbf{S}$ in $\mathcal{N}(\mathbf{m},\mathbf{S})$.
Actually, by the change of variable $\widetilde{\mathbf{q}}_{i}=\sqrt{1/I_{\mathbf{S}}(\mathbf{m},\mathbf{q})}(\mathbf{q}_{i}-\mathbf{c}(\mathbf{m},\mathbf{q}))$
it can be shown that $\widetilde{\mathbf{q}}=(\mathbf{\widetilde{q}}_{1}^{T},\ldots,\widetilde{\mathbf{q}}_{n}^{T})^{T}$
is a $\text{BC}(\mathbf{S})$ with the parameter $\widetilde{\lambda}=U_{n}(\mathbf{m},\mathbf{\widetilde{q}})>0$,
and has the center of mass $\widetilde{\mathbf{c}}(\mathbf{m},\widetilde{\mathbf{q}})=0$
and the $\mathbf{S}$-weighted moment of inertia $I_{\mathbf{S}}(\mathbf{m},\widetilde{\mathbf{q}})=1$;
thus, $\widetilde{\mathbf{q}}\in\text{\ensuremath{\mathcal{N}(\mathbf{m},\mathbf{S})}}$. 

According to Moeckel (2014a), for a positive definite symmetric $2\times2$
matrix $\mathbf{S}$, a configuration $\mathbf{q}$ is a $\text{BC}(\mathbf{S})$
if and only if its corresponding normalized configuration $\widetilde{\mathbf{q}}\in\mathcal{N}(\mathbf{m},\mathbf{S})$
is a critical point of $\widetilde{U}_{n}=U_{n}|_{\mathcal{N}(\mathbf{S})}:\mathcal{N}(\mathbf{S})\rightarrow\mathbb{R}$.
The nullity at a critical point is defined as $\text{null}(\widetilde{\mathbf{q}}):=\dim(\ker(H(\widetilde{\mathbf{q}})))$,
where $\mathbf{H}(\widetilde{\mathbf{q}})$ is the Hessian quadratic
form of $\widetilde{U}_{n}$ on $T_{\widetilde{\mathbf{q}}}\mathcal{N}(\mathbf{m},\mathbf{S})$.
Note that the Hessian $\mathbf{H}(\widetilde{\mathbf{q}})$ of $\widetilde{U}_{n}:\mathcal{N}(\mathbf{m},\mathbf{S})\rightarrow\mathbb{R}$
at a critical point $\widetilde{\mathbf{q}}\in\text{Crit}(\widetilde{U}_{n})$
is given by $\mathbf{H}(\widetilde{\mathbf{q}})\mathbf{v}=\mathbf{v}^{T}\mathbf{H}(\widetilde{\mathbf{q}})\mathbf{v}$,
where 
\begin{equation}
\mathbf{H}(\widetilde{\mathbf{q}})=D^{2}U_{n}(\widetilde{\mathbf{q}})+U_{n}(\mathbf{m},\widetilde{\mathbf{q}})\widehat{\mathbf{S}}\mathbf{M}\label{eq:A3}
\end{equation}
and 
\[
\widehat{\mathbf{S}}=\underbrace{\left(\begin{array}{ccc}
\mathbf{S} & \cdots & 0\\
\vdots & \ddots & \vdots\\
0 & \cdots & \mathbf{S}
\end{array}\right)}_{n\text{ blocks}},\quad\mathbf{M}=\left(\begin{array}{cccccc}
m_{1}\\
 & m_{1} &  &  & \boldsymbol{0}\\
 &  &  & \ddots\\
 & \boldsymbol{0} &  &  & m_{n}\\
 &  &  &  &  & m_{n}
\end{array}\right).
\]

In the case of central configurations, the normalized configuration
space $\mathcal{N}(\mathbf{m},\mathbf{S})$ and the Newtonian force
function $\widetilde{U}_{n}$ are invariant under the diagonal $\text{O}(2)$-action.
Hence, $\widetilde{U}_{n}$ descends to a function $\widehat{U}_{n}:\mathcal{N}(\mathbf{m},\mathbf{S})/\text{O}(2)\rightarrow\mathbb{R}$.
For a central configuration $\widetilde{\mathbf{q}}\in\mathcal{N}(\mathbf{m},\mathbf{S})$,
$\text{null}(\widetilde{\mathbf{q}})\geq1$ and the equivalence class
$[\widetilde{\mathbf{q}}]\in\mathcal{N}(\mathbf{m},\mathbf{S})/\text{O}(2)$
is a critical point of $\widehat{U}_{n}$. Consequently, the Hessian
$\widehat{\mathbf{H}}([\widehat{\mathbf{q}}])$ of $\widehat{U}_{n}$
at $[\widetilde{\mathbf{q}}]$ is obtained by descending $\mathbf{H}(\widetilde{\mathbf{q}})$
to the space $T_{[\widetilde{\mathbf{q}}]}\left(\mathcal{N}(\mathbf{m},\mathbf{S})/\text{O}(2)\right)$. 

In Moeckel (2014a) and Doicu et al. (2020), the non-degeneracy of
a critical point is defined as follows. 
\begin{defn}
\label{def:Let-.-A}Let $\mathbf{q}$ be a $\text{BC}(\mathbf{S})$
with $\mathbf{S}=\text{diag}(\sigma_{x},\sigma_{y})$. Then
\begin{lyxlist}{00.00.0000}
\item [{Case$\ $1:}] for $\sigma_{x}=\sigma_{y}$, the configuration $\mathbf{q}$
is called non-degenerate if the Hessian $\widehat{\mathbf{H}}([\widetilde{\mathbf{q}}])$
is non-degenerate, while 
\item [{Case$\ $2:}] for $\sigma_{x}\not=\sigma_{y}$, the configuration
$\mathbf{q}$ is called non-degenerate if the Hessian $\mathbf{H}(\widetilde{\mathbf{q}})$
is non-degenerate, 
\end{lyxlist}
\end{defn}

where $\widetilde{\mathbf{q}}$ represents the corresponding normalized
configuration of $\mathbf{q}$.
\begin{rem}
\label{rem:By-following-Moczurad}A more convenient way to define
non-degenerateness of central and balanced configurations was suggested
in Moczurad and Zgliczynski (2019) and Moeckel (2014a). In the case
$\sigma_{x}=\sigma_{y}$ a normalized central configuration $\widetilde{\mathbf{q}}$
is called non-degenerate if the matrix $\mathbf{H}(\widetilde{\mathbf{q}})=D^{2}U_{n}(\widetilde{\mathbf{q}})+\lambda\widehat{\mathbf{S}}\mathbf{M}$,
where $\lambda=U_{n}(\mathbf{m},\widetilde{\mathbf{q}})$, is of rank
$2n-1$, while in the case $\sigma_{x}\not=\sigma_{y}$, a normalized
configuration $\widetilde{\mathbf{q}}$ is called non-degenerate if
the matrix $\mathbf{H}(\widetilde{\mathbf{q}})=D^{2}U_{n}(\widetilde{\mathbf{q}})+\lambda\widehat{\mathbf{S}}\mathbf{M}$
is of full rank $2n$. It is easy to check that both definitions are
equivalent.
\end{rem}

In summary, from Eqs. (\ref{eq:DU}) and (\ref{eq:1}) in conjunction
with $\widetilde{\lambda}=U_{n}(\mathbf{m},\mathbf{\widetilde{q}})$,
$\widetilde{\mathbf{c}}(\mathbf{m},\widetilde{\mathbf{q}})=0$, and
$I_{\mathbf{S}}(\mathbf{m},\widetilde{\mathbf{q}})=1$, we infer that
the position vectors of a balanced configuration $\text{BC}(\mathsf{\mathbf{S}})$
satisfies the relative equilibrium equations
\begin{equation}
\mathbf{f}_{i}^{(n)}(\mathbf{m},\widetilde{\mathbf{q}}):=\sum_{\substack{j=1\\
j\not=i
}
}^{n}\frac{m_{j}}{||\widetilde{\mathbf{q}}_{j}-\widetilde{\mathbf{q}}_{i}||^{3}}(\widetilde{\mathbf{q}}_{j}-\widetilde{\mathbf{q}}_{i})+U_{n}(\mathbf{m},\mathbf{\widetilde{q}})\mathbf{S}\widetilde{\mathbf{q}}_{i}=\boldsymbol{0},\label{eq:Alex1}
\end{equation}
for all $i=1,\ldots,n$. In the following we will deal only with normalized
configurations and renounce on the tilde character ``\textasciitilde ''.

A stochastic optimization algorithm for analyzing planar central and
balanced configurations in the $n$-body problem, was designed in
Doicu et al. (2020). This numerical approach is a modified version
of the Minfinder method of Tsoulos and Lagaris (2006) and is devoted
to the solution of the generic system of nonlinear equations 
\begin{equation}
\mathbf{f}(\mathbf{q})=\boldsymbol{0},\label{eq:MinFinder}
\end{equation}
where $\mathbf{q}\in\mathbb{R}^{N}$ is assumed to lie in the box
$B=[a_{1},b_{1}]\times[a_{2},b_{2}]\ldots\times[a_{N},b_{N}]\subset\mathbb{R}^{N}$,
$\mathbf{f}(\mathbf{q})=(f_{1}(\mathbf{q}),f_{2}(\mathbf{q}),\ldots,f_{M}(\mathbf{q}))^{T}$,
$M\geq N$, and $f_{i}:B\rightarrow\mathbb{R}$ are continuous functions.
Actually, the solution of the system of equations (\ref{eq:MinFinder})
is equivalent to the solution of an optimization problem consisting
in the computation of all local minima of the objective function $F:B\subset\mathbb{R}^{N}\rightarrow\mathbb{R}$
given by
\[
F(\mathbf{q})=\frac{1}{2}||\mathbf{f}(\mathbf{q})||^{2}.
\]
The stochastic optimization approach is illustrated in Algorithm \ref{alg:AlgMinFinder},
where $\mathcal{L}$ is a deterministic local optimization method,
while $\mathcal{L}(\mathbf{s})$ is the point where the local search
procedure $\mathcal{L}$ terminates when started at point $\mathbf{s}$. 

\begin{algorithm}
\caption{The main steps of the stochastic optimization method. \label{alg:AlgMinFinder}}

$\bullet$ Initialize the set of distinct solutions $Q=\textrm{Ø}$.

$\bullet$ Generate a set $S=\{\mathbf{s}_{k}\}_{k=1}^{N_{\textrm{s}}}$
of $N_{\textrm{s}}$ sample points in the box $B$.

\textbf{For} $k=1,\ldots,N_{\mathrm{s}}$ \textbf{do}

\hspace{0.5cm}$\bullet$ Initialize the number of distinct solutions
at step $k$,

\hspace{0.5cm}$N_{\mathrm{sol}}(k)=|Q|$.

\hspace{0.5cm}$\mathbf{s}=\mathbf{s}_{k}$.

\hspace{0.5cm}\textbf{If} $\mathbf{s}$ is a start point \textbf{then}

\hspace{1cm}$\bullet$ Start a local search $\mathbf{q}=\mathcal{L}(\mathbf{s})$.

\hspace{1cm}$\bullet$ \textbf{If} $\mathbf{q}\notin Q$, insert
$\mathbf{q}$ in the set of distinct solutions $Q$ and 

\hspace{1cm}update $N_{\mathrm{sol}}(k)\leftarrow N_{\mathrm{sol}}(k)+1$.

\hspace{0.5cm}\textbf{End if}

\hspace{0.5cm}\textbf{If }$N_{\mathrm{sol}}(k)$ does not change
within a prescribed 

\hspace{0.5cm}number of iteration steps $k^{\star}$ \textbf{exit}

\textbf{End for}
\end{algorithm}

The following key elements of the algorithm can be emphasized.
\begin{enumerate}
\item \emph{Generation of sampling points}. A sampling method should create
data that accurately represent the underlying function and preserve
the statistical characteristics of the complete dataset. The following
sampling methods are implemented:\emph{ }(i) pseudo-random number
generators (Marsaglia and Tsang 2000; Matsumoto and Nishimura 1998),
(ii) chaotic method (Dong et al. 2012; Gao and Wang 2007; Gao and
Liu 2012), (iii) low discrepancy method including Halton, Sobol, Niederreiter,
Hammersley, and Faure sequences, (iv) Latin hypercube (McKay et al.
1979), (v) quasi-oppositional differential evolution (Rahnamayan et
al. 2006; Rahnamayan et al. 2008), and (vi) centroidal Voronoi tessellation
(Du et al. 2010). 
\item \emph{Selection of a starting point for the local search}. A point
is considered to be a start point if it is not too close to some already
located minimum or another sample, whereby the closeness with a local
minimum or some other sample is guided through the so-called typical
distance (Tsoulos and Lagaris 2006). 
\item \emph{Local optimization method}. Several optimization software packages
for nonlinear least squares and general function minimization are
implemented. These include (i) the BFGS algorithm of Byrd et al. (1995),
(ii) the TOLMIN algorithm of Powell (1989), (iii) the DQED algorithm
due to Hanson and Krogh (1992), and (iv) the optimization algorithms
implemented in the Portable, Outstanding, Reliable and Tested (PORT)
library. In the latter case, a trust-region method in conjunction
with a Gauss-Newton and a Quasi-Newton model are used to compute the
trial step (Dennis Jr. et al. 1981a; Dennis Jr. et al. 1981b). These
deterministic optimization algorithms can be used in conjunction with
several stochastic solvers, as for example: (i) evolutionary strategy,
(ii) genetic algorithms, and (iii) simulated annealing. 
\item \emph{Stopping rule}. The algorithm must stop when all minima have
been collected with certainty. As default, (i) Bayesian stopping rules
(Zieli\'{n}ski 1981; Boender and Kan 1987; Boender and Romeijn 1995),
and (ii) the double-box stopping rule proposed by Lagaris and Tsoulos
(2008) are implemented. However, because the Bayesian and the double-box
stopping rule are not very efficient for this type of applications
(in order to capture a large number of solutions, either the tolerances
of the stopping rules should be very small or the number of local
searches should be extremely large) we adopted an additional termination
criterion: if the number of solutions does not change within a prescribed
number of iteration steps $k^{\star}$, the algorithm stops.
\end{enumerate}
To analyze planar central and balanced configurations for the $n$-body
problem, the stochastic optimization algorithm is used with $N=2n$
and $M=2n$, and is adapted as follows.
\begin{enumerate}
\item According to Eq. (\ref{eq:Alex1}) and for $\mathbf{S}=\text{diag}(\sigma_{x},\sigma_{y})$,
the functions that determine the objective function $F(\mathbf{q})$
are 
\begin{align}
f_{2i-1}^{(n)}(\mathbf{m},\mathbf{q}) & =\sum_{\substack{j=1\\
j\not=i
}
}^{n}m_{j}\frac{x_{j}-x_{i}}{||\mathbf{q}_{j}-\mathbf{q}_{i}||^{3}}+U_{n}(\mathbf{m},\mathbf{q})\sigma_{x}x_{i},\label{eq:Fx}\\
f_{2i}^{(n)}(\mathbf{m},\mathbf{q}) & =\sum_{\substack{j=1\\
j\not=i
}
}^{n}m_{j}\frac{y_{j}-y_{i}}{||\mathbf{q}_{j}-\mathbf{q}_{i}||^{3}}+U_{n}(\mathbf{m},\mathbf{q})\sigma_{y}y_{i},\label{eq:Fy}
\end{align}
for $i=1,\ldots,n$. 
\item For $\mathbf{q}_{i}=(x_{i},y_{i})^{T}$ and in view of the normalization
condition for the moment of inertia $I_{\mathbf{S}}(\mathbf{m},\mathbf{q})=1$,
i.e., $\sum_{i=1}^{n}m_{i}\mathbf{q}_{i}^{T}\mathbf{S}\mathbf{q}_{i}=1$,
the following simple bounds on the variables: 
\begin{equation}
-l_{xi}\leq x_{i}\leq l_{xi},\,\,\,-l_{yi}\leq y_{i}\leq l_{yi}\label{eq:SB1}
\end{equation}
with 
\begin{equation}
l_{xi}=\frac{1}{\sqrt{m_{i}\sigma_{x}}},\,\,\,l_{yi}=\frac{1}{\sqrt{m_{i}\sigma_{y}}},\label{eq:SB2}
\end{equation}
are imposed.
\item To specify the set of distinct solutions $Q$, we take into account
that for a central configuration, if $\mathbf{q}$ is a solution,
then any (i) permuted solution $\mathcal{P}\mathbf{q}$, (ii) rotated
solution of angle $\alpha$, $\mathcal{R}_{\alpha}\mathbf{q}$, and
(iii) conjugated solutions $\mathcal{C}_{x}\mathbf{q}$ and $\mathcal{C}_{y}\mathbf{q}$
are also solutions. Here, $\mathcal{P}$ and $\mathcal{R}_{\alpha}$
are the permutation and the rotation operator of angle $\alpha$,
respectively, while $\mathcal{C}_{x}$ and $\mathcal{C}_{y}$ stand
for the reflection operators with respect to the $x$- and $y$-axis,
respectively.\textcolor{red}{{} }For a balanced configuration, if $\mathbf{q}$
is a solution, then (i) any permuted solution, (ii) a solution rotated
by $\alpha=\pi$, and (iii) any conjugated solutions are also solutions. 
\item The decision that a solution $\mathbf{q}$, computed by means of a
local optimization method will be included in the set of (distinct)
solutions $Q=\{\mathbf{q}_{i}\}_{i=1}^{N_{\textrm{sol}}}$ is taken
according to the following rule: if (i) the objective function at
$\mathbf{q}$ is smaller than a prescribed tolerance and (ii) the
ordered set of mutual distances $\{R_{ij}\}$ corresponding to $\mathbf{q}$
does not coincides with the ordered set of mutual distances $\{R_{ij}^{\prime}\}$
corresponding to any $\mathbf{q}'\in Q$, then $\mathbf{q}$ is inserted
in the set of solutions $Q$. 
\item In the post-processing stage, several solution tests have been incorporated.
These are related to the fulfillment of the normalization condition
for the moment of inertia, the center-of-mass equation, the Albouy-Chenciner
equations (Albouy and Chenciner 1998), the Morse equality, and the
uniqueness of the solutions. Note that the solution uniqueness is
checked by means of an approach based on the Krawczyk operator method
(Lee and Santoprete 2009; Moczurad and Zgliczynski 2019, 2020). 
\end{enumerate}

\section{Planar central and balanced configurations in the $(n+1)$-body problem
with a small mass\label{subsec:Balanced-configurations-for}}

Let $(\mathbf{q}_{1}^{T},\ldots,\mathbf{q}_{n}^{T},\mathbf{q}_{n+1}^{T})^{T}$
be a normalized $\text{BC}(\mathsf{\mathbf{S}})$ for the $(n+1)$-body
problem with masses $(m_{1},\ldots,m_{n},m_{n+1})^{T}$, that is,
the configuration satisfies the relative equilibrium equations 
\begin{align}
\mathbf{f}_{i}^{(n+1)}(\mathbf{m},m_{n+1},\mathbf{q},\mathbf{q}_{n+1}): & =\sum_{\substack{j=1\\
j\not=i
}
}^{n+1}\frac{m_{j}}{||\mathbf{q}_{j}-\mathbf{q}_{i}||^{3}}(\mathbf{q}_{j}-\mathbf{q}_{i})\nonumber \\
 & +U_{n+1}(\mathbf{m},m_{n+1},\mathbf{q},\mathbf{q}_{n+1})\mathbf{S}\mathbf{q}_{i}\nonumber \\
 & =\boldsymbol{0},\label{eq:N1c}
\end{align}
for all $i=1,\ldots,n+1$, where $\mathbf{m}=(m_{1},\ldots,m_{n})^{T}$,
$\mathbf{q}=(\mathbf{q}_{1}^{T},\ldots,\mathbf{q}_{n}^{T})^{T}$,
and 
\begin{equation}
U_{n+1}(\mathbf{m},m_{n+1}\mathbf{q},\mathbf{q}_{n+1})=U_{n}(\mathbf{m},\mathbf{q})+\sum_{i=1}^{n}\frac{m_{i}m_{n+1}}{||\mathbf{q}_{n+1}-\mathbf{q}_{i}||}.\label{eq:N1d}
\end{equation}
In particular, for $i=1,\ldots,n$, the system of equations (\ref{eq:Alex1})
reads as 

\begin{align}
\mathbf{f}_{i}^{(n+1)}(\mathbf{m},m_{n+1},\mathbf{q},\mathbf{q}_{n+1}):= & \frac{m_{n+1}}{||\mathbf{q}_{n+1}-\mathbf{q}_{i}||^{3}}(\mathbf{q}_{n+1}-\mathbf{q}_{i})+\sum_{\substack{j=1\\
j\not=i
}
}^{n}\frac{m_{j}}{||\mathbf{q}_{j}-\mathbf{q}_{i}||^{3}}(\mathbf{q}_{j}-\mathbf{q}_{i})\nonumber \\
 & +U_{n+1}(\mathbf{m},m_{n+1},\mathbf{q},\mathbf{q}_{n+1})\mathbf{S}\mathbf{q}_{i}\nonumber \\
 & =\boldsymbol{0},\label{eq:N1a}
\end{align}
while for $i=n+1$, we have 
\begin{align}
\mathbf{f}_{n+1}^{(n+1)}(\mathbf{m},m_{n+1},\mathbf{q},\mathbf{q}_{n+1}): & =\sum_{j=1}^{n}\frac{m_{j}}{||\mathbf{q}_{j}-\mathbf{q}_{n+1}||^{3}}(\mathbf{q}_{j}-\mathbf{q}_{n+1})\nonumber \\
 & +U_{n+1}(\mathbf{m},m_{n+1}\mathbf{q},\mathbf{q}_{n+1})\mathbf{S}\mathbf{q}_{n+1}\nonumber \\
 & =\boldsymbol{0}.\label{eq:N1b}
\end{align}

The problem that we intend to solve consists in the computation of
the balanced and central configurations when the masses $m_{1},\ldots,m_{n},m_{n+1}$
are specified and $m_{n+1}\ll m_{i}$ for all $i=1,\ldots,n$. 

A first option is to solve the $(n+1)$-body problem directly by using
the above stochastic optimization algorithm. However in this case,
the standard algorithm should be slightly changed. The reason for
this change is that the bounds $l_{x,n+1}$ and $l_{y,n+1}$ computed
as in Eq. (\ref{eq:SB2}) are very large (because $m_{n+1}$ is very
small), and so, a large number of sampling points is required to find
as many solutions as possible. Instead, the choice $l_{x,n+1}=2\max_{i=1,\ldots,n}\{l_{xi}\}$
and $l_{y,n+1}=2\max_{i=1,\ldots,n}\{l_{yi}\}$ leads to a substantial
reduction of the computational time. 

A second option is to design a numerical algorithm relying on an analytic-continuation
result. Let $\mathbf{q}=(\mathbf{q}_{1}^{T},\ldots,\mathbf{q}_{n}^{T})^{T}$
be a normalized $\text{BC}(\mathsf{\mathbf{S}})$ with masses $\mathbf{m}=(m_{1},\ldots,m_{n})^{T}$,
i.e., the configuration $\mathbf{q}$ satisfies the relative equilibrium
equations 
\begin{equation}
\mathbf{f}_{i}^{(n)}(\mathbf{m},\mathbf{q})=\boldsymbol{0},\label{eq:X0}
\end{equation}
for all $i=1,\ldots,n$. Consider the function 
\begin{equation}
V_{\mathbf{S}n}(\mathbf{m},\mathbf{q},\mathbf{q}_{n+1})=\sum_{j=1}^{n}\frac{m_{j}}{||\mathbf{q}_{j}-\mathbf{q}_{n+1}||}+\frac{1}{2}U_{n}(\mathbf{m},\mathbf{q})\mathbf{q}_{n+1}^{T}\mathbf{S}\mathbf{q}_{n+1}.\label{eq:X1}
\end{equation}

\begin{defn}
A critical point $\mathbf{q}_{n+1}$ of $V_{\mathbf{S}n}(\mathbf{m},\mathbf{q},\cdot)$,
i.e. 
\begin{equation}
\nabla_{\mathbf{q}_{n+1}}V_{\mathbf{S}n}(\mathbf{m},\mathbf{q},\mathbf{q}_{n+1})=\sum_{j=1}^{n}\frac{m_{j}}{||\mathbf{q}_{j}-\mathbf{q}_{n+1}||^{3}}(\mathbf{q}_{j}-\mathbf{q}_{n+1})+U_{n}(\mathbf{m},\mathbf{q})\mathbf{S}\mathbf{q}_{n+1}=\boldsymbol{0},\label{eq:X2}
\end{equation}
is called a $\text{BC}(\mathsf{\mathbf{S}})$ of the restricted $(n+1)$-body
problem. A $\text{BC}(\mathsf{\mathbf{S}})$ of the restricted $(n+1)$-body
problem $\mathbf{q}_{n+1}$ is called non-degenerate if $D_{\mathbf{q}_{n+1}}^{2}V_{\mathbf{S}n}(\mathbf{m},\mathbf{q},\mathbf{q}_{n+1})$
is of full rank $2$. 
\end{defn}

From Eqs. (\ref{eq:N1d})--(\ref{eq:N1b}), it is readily seen that
$U_{n+1}(\mathbf{m},0,\mathbf{q},\mathbf{q}_{n+1})=U_{n}(\mathbf{m},\mathbf{q})$,
and consequently, if $\mathbf{q}$ is a normalized $\text{BC}(\mathsf{\mathbf{S}})$
with masses $\mathbf{m}$, and $\mathbf{q}_{n+1}$ is a $\text{BC}(\mathsf{\mathbf{S}})$
of the restricted $(n+1)$-body problem, that
\begin{align}
\mathbf{f}_{i}^{(n+1)}(\mathbf{m},0,\mathbf{q},\mathbf{q}_{n+1}) & =\mathbf{f}_{i}^{(n)}(\mathbf{m},\mathbf{q})=\boldsymbol{0},\,\,\,i=1,\ldots,n,\label{eq:A1}\\
\mathbf{f}_{n+1}^{(n+1)}(\mathbf{m},0,\mathbf{q},\mathbf{q}_{n+1}) & =\nabla_{\mathbf{q}_{n+1}}V_{\mathbf{S}n}(\mathbf{m},\mathbf{q},\mathbf{q}_{n+1})=\boldsymbol{0}.\label{eq:A2}
\end{align}
The next result is a simplified statement of Proposition 1 given in
Xia (1991) adapted to the BC case. 
\begin{prop}
\label{prop:Prop4}Let $\boldsymbol{\mathbf{q}}_{0}=(\mathbf{q}_{01}^{T},\ldots,\mathbf{q}_{0n}^{T})^{T}$
be a non-degenerate normalized $\text{BC}(\mathsf{\mathbf{S}})$ for
the $n$-body problem with masses $\mathbf{m}=(m_{1},\ldots,m_{n})^{T}$
and $\mathbf{q}_{0,n+1}$ a non-degenerate $\text{BC}(\mathsf{\mathbf{S}})$
for the restricted $(n+1)$-body problem\emph{.} Then the configuration
$(\boldsymbol{\mathbf{q}}_{0}^{T},\boldsymbol{\mathbf{q}}_{0,n+1}^{T})^{T}=(\mathbf{q}_{01}^{T},\ldots,\mathbf{q}_{0n}^{T},\mathbf{q}_{0,n+1}^{T})^{T}$
for the masses $(\mathbf{m}^{T},0)^{T}=(m_{1},\ldots,m_{n},0)^{T}$
can be analytically continued to a normalized $\text{BC}(\mathsf{\mathbf{S}})$
for the $(n+1)$-body problem for any mass $m_{n+1}$ in an open neighborhood
of $0$. 
\end{prop}

\begin{proof}
By hypothesis, we have that $\mathbf{f}_{i}^{(n)}(\mathbf{m},\mathbf{q}_{0})=\boldsymbol{0}$
for all $i=1,\ldots,n$, and $\nabla_{\mathbf{q}_{n+1}}V_{\mathbf{S}n}(\mathbf{m},\mathbf{q}_{0},\mathbf{q}_{0,n+1})=\boldsymbol{0}$.
The aim is to show that for $(\mathbf{m}^{T},m_{n+1})^{T}$ with $m_{n+1}\in U_{+}(0)$,
where $U_{+}(0)\subset\mathbb{R}_{+}$ is an open neighborhood of
$0$, there exist $\mathbf{q}$ and $\mathbf{q}_{n+1}$ in the neighborhoods
of $\mathbf{q}_{0}$ and $\mathbf{q}_{0,n+1}$, respectively, such
that $\mathbf{f}_{i}^{(n+1)}(\mathbf{m},m_{n+1},\mathbf{q},\mathbf{q}_{n+1})=\boldsymbol{0}$
for all $i=1,\ldots,n+1$. For 
\[
\mathbf{f}_{i}^{(n+1)}(\mathbf{m},m_{n+1},\mathbf{q},\mathbf{q}_{n+1}):=\sum_{\substack{j=1\\
j\not=i
}
}^{n+1}\frac{m_{j}}{||\mathbf{q}_{j}-\mathbf{q}_{i}||^{3}}(\mathbf{q}_{j}-\mathbf{q}_{i})+\lambda\mathbf{S}\mathbf{q}_{i}=\mathbf{0},
\]
with $\lambda=U_{n+1}(\mathbf{m},0,\mathbf{q},\mathbf{q}_{n+1})=U_{n}(\mathbf{m},\mathbf{q})$
fixed, consider the function 
\[
\mathbf{f}(m_{n+1},\mathbf{q},\mathbf{q}_{n+1}):\mathbb{R}_{+}\times\mathbb{R}^{2n}\times\mathbb{R}^{2}\rightarrow\mathbb{R}^{2n+2},
\]
defined by 
\begin{align}
\mathbf{f}(m_{n+1},\mathbf{q},\mathbf{q}_{n+1}) & =\left(\begin{array}{c}
\mathbf{f}_{1}^{(n+1)}(\mathbf{m},m_{n+1},\mathbf{q},\mathbf{q}_{n+1})\\
\vdots\\
\mathbf{f}_{n+1}^{(n+1)}(\mathbf{m},m_{n+1},\mathbf{q},\mathbf{q}_{n+1})
\end{array}\right).\label{eq:1-6}
\end{align}
From Eqs. (\ref{eq:A1})--(\ref{eq:A2}), we see that $\mathbf{f}_{i}^{(n+1)}(\mathbf{m},0,\mathbf{q}_{0},\mathbf{q}_{0,n+1})=\boldsymbol{0}$,
$i=1,\ldots,n$, and $\mathbf{f}_{n+1}^{(n+1)}(\mathbf{m},0,\mathbf{q}_{0},\mathbf{q}_{0,n+1})=\boldsymbol{0}$;
hence, $\mathbf{f}(0,\mathbf{q}_{0},\mathbf{q}_{0,n+1})=\boldsymbol{0}$.
Moreover, for $\sigma_{x}\neq\sigma_{y}$, it can be checked using
the non-degeneracy assumption together with Remark \ref{rem:By-following-Moczurad}
that the Jacobian $D_{(\mathbf{q},\mathbf{q}_{n+1})}\mathbf{f}(0,\mathbf{q}_{0},\mathbf{q}_{0,n+1})$
\begin{align*}
 & D_{(\mathbf{q},\mathbf{q}_{n+1})}\mathbf{f}(0,\mathbf{q}_{0},\mathbf{q}_{0,n+1})\\
 & =\left(\begin{array}{cc}
D_{\mathbf{q}}\mathbf{f}_{1}^{(n)}(\mathbf{m},\mathbf{q}_{0}) & \boldsymbol{0}\\
\vdots & \vdots\\
D_{\mathbf{q}}\mathbf{f}_{n}^{(n)}(\mathbf{m},\mathbf{q}_{0}) & \boldsymbol{0}\\
D_{\mathbf{q}}\mathbf{f}_{n+1}^{(n+1)}(\mathbf{m},0,\mathbf{q}_{\text{0}},\mathbf{q}_{0,n+1}) & (D_{\mathbf{q}_{n+1}}^{2}V_{\mathbf{S}n})(\mathbf{m},\mathbf{q}_{0},\mathbf{q}_{0,n+1})
\end{array}\right)
\end{align*}
is non-singular, i.e., $\textrm{rank}(D_{(\mathbf{q},\mathbf{q}_{n+1})}\mathbf{f}(0,\mathbf{q}_{0},\mathbf{q}_{0,n+1}))=2n+2$.
In this regard, according to the implicit function theorem, there
exist an open neighborhood of $0$, $U_{+}(0)\subset\mathbb{R}_{+}$,
and a function
\[
\mathbf{g}:U_{+}(0)\rightarrow\mathbb{R}^{2n}\times\mathbb{R}^{2}
\]
such that $(\mathbf{q}_{0},\mathbf{q}_{0,n+1})=\mathbf{g}(0)$ and
\[
\mathbf{f}(m_{n+1},\mathbf{g}(m_{n+1}))=\boldsymbol{0}\textrm{ for all }m_{n+1}\in U_{+}(0).
\]
Thus, $\mathbf{f}_{i}^{(n+1)}(\mathbf{m},m_{n+1},\mathbf{g}(m_{n+1})=\boldsymbol{0}$
for all $m_{n+1}\in U_{+}(0)$ and $i=1,\ldots,n+1$. 

In the case of central configurations ($\sigma_{x}=\sigma_{y}$),
a direct application of the implicit function theorem has to involve
an additional reduction procedure, sinces the rotational direction
is the kernel of $D_{(\mathbf{q},\mathbf{q}_{n+1})}\mathbf{f}(0,\mathbf{q}_{0},\mathbf{q}_{0,n+1})$.
Instead, here we apply an equivariant version of the implicit function
theorem as in (Bettiol et al. (2014)). To prove the Proposition, we
first consider the $(n+1)$th equation of the $(n+1)$-body problem,
i.e., 
\[
\mathbf{f}_{n+1}^{(n+1)}(\mathbf{m},m_{n+1},\mathbf{q},\mathbf{q}_{n+1})=\sum_{j=1}^{n}\frac{m_{j}}{||\mathbf{q}_{j}-\mathbf{q}_{n+1}||^{3}}(\mathbf{q}_{j}-\mathbf{q}_{n+1})+\lambda\mathbf{S}\mathbf{q}_{n+1}.
\]
Since the Jacobian of $\mathbf{f}_{n+1}^{n+1}$ with respect to $\mathbf{q}_{n+1}$
is assumed to be non-degenerate, the implicit function theorem provides
a unique real analytic map $\mathbf{g}_{0}:U_{+}(0)\times U(\mathbf{q}_{0})\rightarrow\mathbb{R}^{2}$,
where $U_{+}(0)\subset\mathbb{R}_{+}$ and $U(\mathbf{q}_{0})\subset\mathbb{R}^{2n}$
are neighborhoods of $0$ and $\mathbf{q}_{0}$, respectively, such
that $\mathbf{g}_{0}(0,\mathbf{q}_{0})=\mathbf{q}_{0,n+1}$ and 
\begin{equation}
\mathbf{f}_{n+1}^{(n+1)}(\mathbf{m},m_{n+1},\mathbf{q},\mathbf{g}_{0}(\mathbf{q}))=0,\label{eq:1-1}
\end{equation}
or equivalently, 
\begin{equation}
\nabla_{\mathbf{q}_{n+1}}U_{n+1}(\mathbf{m},m_{n+1},\mathbf{q},\mathbf{g}_{0}(m_{n+1},\mathbf{q}))+\frac{\lambda}{2}\nabla_{\mathbf{q}_{n+1}}I_{n+1}(\mathbf{m},m_{n+1},\mathbf{q},\mathbf{g}_{0}(m_{n+1},\mathbf{q}))=0,\label{eq:1-2}
\end{equation}
where 
\[
I_{n+1}(\mathbf{m},m_{n+1},\mathbf{q},\mathbf{q}_{n+1})=\sum_{i=1}^{n}m_{i}\mathbf{q}_{i}^{T}\mathbf{q}_{i}+m_{n+1}\mathbf{q}_{n+1}^{T}\mathbf{q}_{n+1}.
\]
Moreover, using the slice theorem it is possible to show that the
map $\mathbf{g}_{0}$ extends to an equivariant map with respect to
the $SO(2)$ diagonal action defined on $U_{+}(0)\times\widehat{U}(\mathbf{q}_{0})$,
where $\widehat{U}(\mathbf{q}_{0})$ is a $SO(2)-$invariant neighborhood
of the orbit $SO(2)\mathbf{q}_{0}$, i.e., for all $O\in SO(2)$ and
all $(m_{n+1},\mathbf{q})\in U_{+}(0)\times\widehat{U}(\mathbf{q}_{0})$,
we have $\mathbf{g}_{0}(m_{n+1},O\mathbf{q})=O\mathbf{g}_{0}(m_{n+1},\mathbf{q})$.
Next, for the first $n$ equations of the $(n+1)$-body problem, we
set
\begin{equation}
\mathbf{f}(m_{n+1},\mathbf{q})=\left(\begin{array}{c}
\mathbf{f}_{1}^{(n+1)}(\mathbf{m},m_{n+1},\mathbf{q},\mathbf{g}_{0}(m_{n+1},\mathbf{q}))\\
\vdots\\
\mathbf{f}_{n}^{(n+1)}(\mathbf{m},m_{n+1},\mathbf{q},\mathbf{g}_{0}(m_{n+1},\mathbf{q}))
\end{array}\right).\label{eq:1-3}
\end{equation}
$\mathbf{f}(m_{n+1},\mathbf{q})$ has an $SO(2)$ symmetry, and therefore,
the kernel of the Jacobian of $\text{\ensuremath{\mathbf{f}}(\ensuremath{m_{n+1}},\ensuremath{\mathbf{q}})}$
with respect to $\mathbf{q}$ is the rotational direction. Because
the implicit function theorem cannot be applied directly, we use an
equivariant version of the implicit function theorem due Bettiol et
al. (2014). For doing this, we consider the function 
\[
\mathfrak{f}:U_{+}(0)\times\widehat{U}(\mathbf{q}_{0})\rightarrow\mathbb{R},
\]
defined by 
\[
\mathfrak{f}(m_{n+1},\mathbf{q}):=U_{n+1}(\mathbf{m},m_{n+1},\mathbf{q},\mathbf{g}_{0}(m_{n+1},\mathbf{q}))+\frac{\lambda}{2}I_{n+1}(\mathbf{m},m_{n+1},\mathbf{q},\mathbf{g}_{0}(m_{n+1},\mathbf{q})).
\]
By Eq. (\ref{eq:1-2}) and the definition of $I_{n+1}$, we have
\begin{align}
 & \nabla_{\mathbf{q}}\mathfrak{f}(m_{n+1},\mathbf{q})\nonumber \\
 & =\nabla_{\mathbf{q}}U_{n+1}(\mathbf{m},m_{n+1},\mathbf{q},\mathbf{g}_{0}(m_{n+1},\mathbf{q}))+\frac{\lambda}{2}\nabla_{\mathbf{q}}I_{n+1}(\mathbf{m},m_{n+1},\mathbf{q},\mathbf{g}_{0}(m_{n+1},\mathbf{q}))\nonumber \\
 & +\left(\frac{\partial\mathbf{g}_{0}}{\partial\mathbf{q}}\right)^{T}\left(\nabla_{\mathbf{q}_{n+1}}U_{n+1}(\mathbf{m},m_{n+1},\mathbf{q},\mathbf{g}_{0}(m_{n+1},\mathbf{q}))\right.\nonumber \\
 & \left.+\frac{\lambda}{2}\nabla_{\mathbf{q}_{n+1}}I_{n+1}(\mathbf{m},m_{n+1},\mathbf{q},\mathbf{g}_{0}(m_{n+1},\mathbf{q}))\right)\nonumber \\
 & =\nabla_{\mathbf{q}}U_{n+1}(\mathbf{m},m_{n+1},\mathbf{q},\mathbf{g}_{0}(m_{n+1},\mathbf{q}))+\frac{\lambda}{2}\nabla_{\mathbf{q}}I_{n+1}(\mathbf{m},m_{n+1},\mathbf{q},\mathbf{g}_{0}(m_{n+1},\mathbf{q}))\label{eq:1-4}
\end{align}
and
\begin{align}
D_{\mathbf{q}}^{2}\mathfrak{f}(m_{n+1},\mathbf{q}) & =D_{\mathbf{q}}^{2}U_{n+1}(\mathbf{m},m_{n+1},\mathbf{q},\mathbf{g}_{0}(m_{n+1},\mathbf{q}))\nonumber \\
 & +\frac{\lambda}{2}D_{\mathbf{q}}^{2}I_{n+1}(\mathbf{m},m_{n+1},\mathbf{q},\mathbf{g}_{0}(m_{n+1},\mathbf{q}))\nonumber \\
 & +m_{n+1}\mathbf{G}(\mathbf{q}),\label{eq:1-5}
\end{align}
where $\mathbf{G}(\mathbf{q})$ is a $2n\times2n$ matrix valued analytic
function. Inspecting Eqs. (\ref{eq:1-3}) and (\ref{eq:1-4}), we
find $\nabla_{\mathbf{q}}\mathfrak{f}(m_{n+1},\mathbf{q})=\mathbf{f}(m_{n+1},\mathbf{q})$,
yielding $D_{\mathbf{q}}^{2}\mathfrak{f}(m_{n+1},\mathbf{q})=D_{\mathbf{q}}\mathbf{f}(m_{n+1},\mathbf{q})$.
The non-degeneracy assumption in conjunction with Eq. (\ref{eq:1-5})
implies that the requirement of the implicit function theorem in (Bettiol
et al. (2014)), that is, 
\[
\ker\left(D_{\mathbf{q}}^{2}\mathfrak{f}(0,\mathbf{q}_{0})\right)=T_{\mathbf{q}_{0}}\left(SO(2)\cdot\mathbf{q}_{0}\right)
\]
is fulfilled. As a result, there exists a real analytic map $\mathbf{g}:U_{+}(0)\rightarrow\mathbb{R}^{2n}$
such that $\mathbf{g}(0)=\mathbf{q}_{0}$ and $(m_{n+1},\mathbf{g}(m_{n+1}),\mathbf{g}_{0}(m_{n+1},\mathbf{g}(m_{n+1})))\in\mathbb{R}_{+}\times\mathbb{R}^{2n}\times\mathbb{R}^{2}$
is a central configuration for all $m_{n+1}\in U_{+}(0)$. 
\end{proof}
From a computational point of view, Proposition \ref{prop:Prop4}
shows that a solution of the $(n+1)$-body problem with given masses
$(\mathbf{m}^{T},m_{n+1})^{T}=(m_{1},\ldots,m_{n},m_{n+1})^{T}$ and
$m_{n+1}\ll m_{i}$ for all $i=1,\ldots,n$, exists in a neighborhood
of a configuration formed by the solutions of the $n$-body problem
and the restricted $(n+1)$-body problem. In other words, for given
masses $(\mathbf{m}^{T},m_{n+1})^{T}=(m_{1},\ldots,m_{n},m_{n+1})^{T}$
with $m_{n+1}\ll m_{i}$ for all $i=1,\ldots,n$, $(\boldsymbol{\mathbf{q}}_{0}^{T},\boldsymbol{\mathbf{q}}_{0,n+1}^{T})^{T}=(\mathbf{q}_{01}^{T},\ldots,\mathbf{q}_{0n}^{T},\mathbf{q}_{0n+1}^{T})^{T}$
can be regarded as an ``approximate'' $\text{BC}(\mathsf{\mathbf{S}})$
for the $(n+1)$-body problem, while an ``exact'' $\text{BC}(\mathsf{\mathbf{S}})$
can be obtained by solving the $(n+1)$-body problem (\ref{eq:N1a})--(\ref{eq:N1b})
in a neighborhood of $(\boldsymbol{\mathbf{q}}_{0}^{T},\boldsymbol{\mathbf{q}}_{0,n+1}^{T})^{T}=(\mathbf{q}_{01}^{T},\ldots,\mathbf{q}_{0n}^{T},\mathbf{q}_{0n+1}^{T})^{T}$. 

\begin{algorithm}
\caption{Analytic-continuation algorithm for computing balanced and central
configurations for the $(n+1)$-body problem with a small mass.\label{alg:Algorithm1}}

\begin{description}
\item [{Step}] 1. Compute all solutions $\mathbf{q}_{0}^{(k)}=(\mathbf{q}_{01}^{(k)T},\ldots,\mathbf{q}_{0n}^{(k)T})^{T}\in\mathbb{R}^{2n}$,
$k=1,\ldots,N_{\textrm{sol}}(n)$ of the $n$-body problem (\ref{eq:Alex1}).
\item [{Step}] 2. For each $k$th solution $\mathbf{q}_{0}^{(k)}$, compute
all solutions $\mathbf{q}_{0n+1}^{(k,l)}\in\mathbb{R}^{2}$, $l=1,\ldots,N_{\textrm{sol}}(k,n)$
of the\emph{ }restricted $(n+1)$-body problem (\ref{eq:X2}).
\item [{Step}] 3. For each initial guess
\[
\mathbf{q}_{0}^{(k,l)}=(\mathbf{q}_{0}^{(k)T},\mathbf{q}_{0n+1}^{(k,l)T})^{T}=(\mathbf{q}_{01}^{(k)T},\ldots,\mathbf{q}_{0n}^{(k)T},\mathbf{q}_{0n+1}^{(k,l)T})^{T}\in\mathbb{R}^{2n+2},
\]
 solve the $(n+1)$-body problem (\ref{eq:N1a})--(\ref{eq:N1b})
for 
\[
\mathbf{q}^{(k,l)}=(\mathbf{q}^{(k)T},\mathbf{q}_{n+1}^{(k,l)T})^{T}=(\mathbf{q}_{1}^{(k)T},\ldots,\mathbf{q}_{n}^{(k)T},\mathbf{q}_{n+1}^{(k,l)T})^{T}
\]
 in a neighborhood of $\mathbf{q}_{0}^{(k,l)}$, e.g., in a box 
\begin{align*}
B_{\delta}(\mathbf{q}_{0}^{(k,l)}) & =\{\mathbf{q}^{(k,l)}\mid|x_{i}^{(k,l)}-x_{0i}^{(k,l)}|\leq\delta|x_{0i}^{(k,l)}|,\\
 & \,|y_{i}^{(k,l)}-y_{0i}^{(k,l)}|\leq\delta|y_{0i}^{(k,l)}|,\,i=1,\ldots,n+1\},
\end{align*}
where $\mathbf{q}_{i}^{(k,l)}=(x_{i}^{(k,l)},y_{i}^{(k,l)})^{T}$,
$\mathbf{q}_{0i}^{(k,l)}=(x_{0i}^{(k,l)},y_{0i}^{(k,l)})^{T}$, and
$\delta$ is sufficiently small.
\end{description}
\end{algorithm}

Algorithm \ref{alg:Algorithm1} is a practical implementation of this
analytic-continuation result. It is fully based on the stochastic
optimization algorithm, which is used in Step 1 to compute all solutions
of the $n$-body problem, and in Step 2 to compute all solutions of
the restricted $(n+1)$-body problem. In the second case, corresponding
to $N=M=2$, we 
\begin{enumerate}
\item impose the simple bounds on the variables (\ref{eq:SB1}), where 
\[
l_{x,n+1}=2\max_{i=1,\ldots,n}\{l_{xi}\},\,\,\,l_{y,n+1}=2\max_{i=1,\ldots,n}\{l_{yi}\},
\]
and $l_{xi}$ and $l_{yi}$ are given by Eq. (\ref{eq:SB2}), 
\item adopt as sampling method, the pseudo-random number generator method,
and 
\item assume for simplicity, that $\mathbf{q}_{1}$ and $\mathbf{q}_{2}$
belong to the set of distinct solutions $Q$ if $\mathbf{q}_{1}\neq\mathbf{q}_{2}$. 
\end{enumerate}
Thus, in principle, we do not exclude from the final set of solutions,
the rotated and reflected solutions. Finally, in Step 3, only the
local search procedure $\mathbf{q}^{(k,l)}=\mathcal{L}(\mathbf{q}_{0}^{(k,l)})$
is applied, and a solution $\mathbf{q}^{(k,l)}$ is accepted if the
objective function is below a prescribed tolerance.

The main peculiarities of the algorithm are that (i) the initial guess
$\mathbf{q}_{0}^{(k,l)}=(\mathbf{q}_{0}^{(k)T},\mathbf{q}_{0n+1}^{(k,l)T})^{T}$
corresponds to the case $m_{n+1}=0$ (hence, it does not depend on
the mass $m_{n+1}$), and (ii) $B_{\delta}(\mathbf{q}_{0}^{(k,l)})$
is a small box around $\mathbf{q}_{0}^{(k,l)}$. As a result,
\begin{enumerate}
\item the number of solutions of the $(n+1)$-body problem
\[
N_{\textrm{sol}}(n+1)=\sum_{k=1}^{N_{\textrm{sol}}(n)}N_{\textrm{sol}}(k,n)
\]
is independent on $m_{n+1}$,
\item the set of solution $\{\mathbf{q}^{(k,l)}\}$ corresponds to a sufficiently
small value of $m_{n+1}$,
\item the solution $\mathbf{q}^{(k,l)}=(\mathbf{q}^{(k)T},\mathbf{q}_{n+1}^{(k,l)T})^{T}\in B_{\delta}(\mathbf{q}_{0}^{(k,l)})$
is close to the initial guess $\mathbf{q}_{0}^{(k,l)}=(\mathbf{q}_{0}^{(k)T},\mathbf{q}_{0n+1}^{(k,l)T})^{T}$,
whereby $\mathbf{q}^{(k)}$ is near the solution of the $n$-body
problem $\mathbf{q}_{0}^{(k)}$, and $\mathbf{q}_{n+1}^{(k,l)}$ is
near the solution of the restricted $(n+1)$-body problem $\mathbf{q}_{0n+1}^{(k,l)}$,
and
\item $||\mathbf{q}^{(k,l)}-\mathbf{q}_{0}^{(k,l)}||\rightarrow0$ as $m_{n+1}\rightarrow0$;
\end{enumerate}
Actually, the solution $\mathbf{q}^{(k,l)}$ belongs to the box $B_{\delta}(\mathbf{q}_{0}^{(k,l)})$,
and there is no guarantee that for example, when $m_{n+1}$ is above
an upper bound $\overline{m}_{n+1}$, there are no other solutions
outside the domain $\cup_{k,l}B_{\delta}(\mathbf{q}_{0}^{(k,l)})$.
However, the analytic-continuation method is more efficient than the
direct method. The reason is that two optimization problems of lower
dimensions are solved consecutively (the dimensions are $(M=2n,N=2n)$
in Step 1 and $(M=2,N=2)$ in Step 2), and therefore, the number of
sampling points required to capture a large number of solutions is
smaller than in the case of the direct method. 

\section{Numerical simulations}

The goal of our numerical analysis is twofold. First, to illustrate
some central and balanced configurations for the $(n+1)$-body problem
with a small mass, and second, to analyze the accuracy and efficiency
of the analytic-continuation algorithm, and in particular, to provide
a numerical verification of the analytic-continuation method, according
to which, a solution of the $(n+1)$-body problem exists in a neighborhood
of a configuration formed by the solutions of the $n$-body problem
and the restricted $(n+1)$-body problem.

In our simulations we choose $m_{i}=m=0.1$, for $i=1,\ldots,n$,
and $m_{n+1}=\varepsilon m$, where the mass parameter $\varepsilon$
takes the values $10^{-8}$, $10^{-9}$, and $10^{-10}$. For central
configurations we set $\sigma_{x}=\sigma_{y}=1.0$, while for balanced
configurations we use $\sigma_{x}=1.0$ and $\sigma_{y}=0.3$. The
parameter $\delta$ specifying the dimension of the box $B_{\delta}(\mathbf{q}_{0}^{(k,l)})$
in Step 3 of Algorithm \ref{alg:Algorithm1} is $5\times10^{-2}$.
For the $n$-body problem, the number of sample points is $N_{\textrm{s}}=10^{6}$,
the number of iteration steps within the number of solutions does
not change is $k^{\star}=200$, and the sampling method is a Faure
sequence. For the restricted $(n+1)$-body problem, the number of
sample points is $N_{\textrm{s}}=5\times10^{4}$. In the case of the
$(n$+1)-body problem solved by the direct method, the number of sample
points is $N_{\textrm{s}}=9\times10^{6}$, the number of iteration
steps within the number of solutions does not change is $k^{\star}=3000$,
and the sampling method is a chaotic method. Note that this large
number of sample points is required in order to capture as many solutions
as possible. 

The deviation of the solution $\mathbf{q}^{(k,l)}$ of the $(n+1)$-body
problem computed by the analytic-continuation method with respect
to the initial guess $\mathbf{q}_{0}^{(k,l)}$, where $k=1,\ldots,N_{\textrm{sol}}(n)$
and $l=1,\ldots,N_{\mathrm{sol}}(k,n)$, is characterized through
the RMS of the absolute error in Cartesian coordinates
\[
\Delta q_{0kl}=\sqrt{\frac{1}{(n+1)}\sum_{i=1}^{n+1}||\mathbf{q}_{i}^{(k,l)}-\mathbf{q}_{0i}^{(k,l)}||^{2},}
\]
and the average RMS 
\[
\Delta q_{0}=\sqrt{\frac{1}{N_{\textrm{sol}}(n)}\sum_{k=1}^{N_{\mathrm{sol}}(n)}\left(\frac{1}{N_{\textrm{sol}}(k,n)}\sum_{l=1}^{N_{\textrm{sol}}(k,n)}\Delta q_{0kl}^{2}\right).}
\]
The deviation of the solution $\mathbf{\widehat{q}}^{(m)}$ of the
$(n+1)$-body problem computed by the direct method with respect to
an analytic-continuation solution $\mathbf{q}^{(k,l)}$, where $m=1,\ldots,\widehat{N}_{\mathrm{sol}}(n+1)$
and $\widehat{N}_{\mathrm{sol}}(n+1)$ is the number of solutions,
is quantified as follows. For each $\mathbf{\widehat{q}}^{(m)}$,
we determine the corresponding solution $\mathbf{q}^{(k_{0},l_{0})}$
computed by the analytic-continuation method as
\[
(k_{0},l_{0})=\arg\min_{(k,l)}\sum_{1\leq i<j\leq n}|\widehat{R}_{ij}^{(m)}-R_{ij}^{(k,l)}|^{2},
\]
and accordingly, calculate the RMS of the absolute error in radial
distances 
\[
\Delta R_{m}=\sqrt{\frac{1}{\mathcal{N}}\sum_{1\leq i<j\leq n}|\widehat{R}_{ij}^{(m)}-R_{ij}^{(k_{0},l_{0})}|^{2}},
\]
and the average RMS 
\[
\Delta R=\sqrt{\frac{1}{\widehat{N}_{\mathrm{sol}}(n+1)}\sum_{m=1}^{\widehat{N}_{\mathrm{sol}}(n+1)}\Delta R_{m}^{2}},
\]
where in general, the $R_{ij}$ are the mutual distances of the configuration
$\mathbf{q}$, and $\mathcal{N}=n(n+1)/2$. 

The results of our numerical analysis are available at website: \emph{https://github.com}\\
\emph{/AlexandruDoicu/Central-and-Balanced-Configurations-with-a-small-mass.}
The simulations were performed on a computer Intel Core x86\_64 CPU
2.70GHz. The output files contain the following data.
\begin{enumerate}
\item Results related to the analytic-continuation method:
\begin{enumerate}
\item the number of solutions of the $n$-body problem $N_{\textrm{sol}}(n)$,
and for each $k$th solution of the $n$-body problem, the number
of solutions of the restricted $(n+1)$-body problem $N_{\textrm{sol}}(k,n)$; 
\item the (total) number of solutions of the $(n+1)$-body problem $N_{\textrm{sol}}(n+1)$
and the number of distinct solutions excluding symmetries $N_{\textrm{sol}}^{0}(n+1)$;
\item for each configuration $\mathbf{q}^{(k,l)}$ of the $(n+1)$-body
problem, where $k=1,\ldots,N_{\textrm{sol}}(n)$ and $l=1,\ldots,N_{\mathrm{sol}}(k,n)$:
(i) the Cartesian coordinates of the point masses, (ii) the residual
of the relative equilibrium equations, and (iii) the RMS of the absolute
error in Cartesian coordinates with respect to the initial guess $\Delta q_{0kl}$; 
\item the residual of the normalization condition for the moment of inertia
and the Cartesian coordinates of the center of mass.
\end{enumerate}
\item Results related to the direct method:
\begin{enumerate}
\item the number of solutions of the $(n+1)$-body problem $\widehat{N}_{\mathrm{sol}}(n+1)$;
\item for each configuration $\mathbf{\widehat{q}}^{(m)}$ of the $(n+1)$-body
problem, where $m=1,\ldots,\widehat{N}_{\mathrm{sol}}(n+1)$: (i)
the Cartesian coordinates of the point masses, (ii) the residual of
the relative equilibrium equations, and (iii) the RMS of the absolute
error in radial distances with respect to the analytic-continuation
solution $\Delta R_{m}$;
\end{enumerate}
\item the average RMS of the absolute errors in Cartesian coordinates and
mutual distances, and the computational times.
\end{enumerate}

\subsection{Central configurations}

The central configurations computed by the analytic-continuation method
in the cases $n=4$, $5$, and $6$ are illustrated in Figs. \ref{fig:CC4},
\ref{fig:CC5}, and \ref{fig:CC6}, respectively. The following features
are apparent.
\begin{enumerate}
\item If a configuration $\mathbf{q}_{0}^{(k)}$ of the $n$-body problem
has an axis of symmetry (reflection axis), the corresponding configuration
$\mathbf{q}^{(k,l)}$ of the $(n+1)$-body problem inherits this symmetry.
\item The solutions $\mathbf{q}^{(k,l)}$ of the $(n+1)$-body problem plotted
in Figs. \ref{fig:CC4}--\ref{fig:CC6} include discrete rotated
and reflected solutions. The distinct central configurations without
these symmetries are illustrated in Figs. \ref{fig:CC4a}--\ref{fig:CC6a}.
Note that these configurations are almost identical with the configurations
delivered by the direct method. 
\end{enumerate}
In Table \ref{tab:RMS}, we provide the numbers of central configurations,
the RMS errors, and the computational times for the analytic-continuation
and the direct method. The following conclusions can be drawn.
\begin{enumerate}
\item For the mass parameter $\varepsilon=10^{-8},10^{-9},10^{-10}$, the
numbers of solutions do not change, and the number of distinct solutions
(excluding discrete rotated and reflected solutions) computed by the
analytic-continuation method coincides with the number of solutions
computed by the direct method, i.e., $N_{\textrm{sol}}^{0}(n+1)=\widehat{N}_{\mathrm{sol}}(n+1)$.
This is an indication that the analytic-continuation algorithm presumably
delivers all solutions of the $(n+1)$-body problem.
\item The RMS error $\Delta q_{0}$ decreases with $\varepsilon$; this
result suggests that $||\mathbf{q}^{(k,l)}-\mathbf{q}_{0}^{(k,l)}||\rightarrow0$
as $m_{n+1}\rightarrow0$.
\item The RMS error $\Delta R$ is small; hence, the solutions corresponding
to the analytic-continuation and the direct method are very close.
\item The analytic-continuation method is on average 7 times faster than
the direct method. 
\end{enumerate}
\begin{figure}
\includegraphics[scale=0.5]{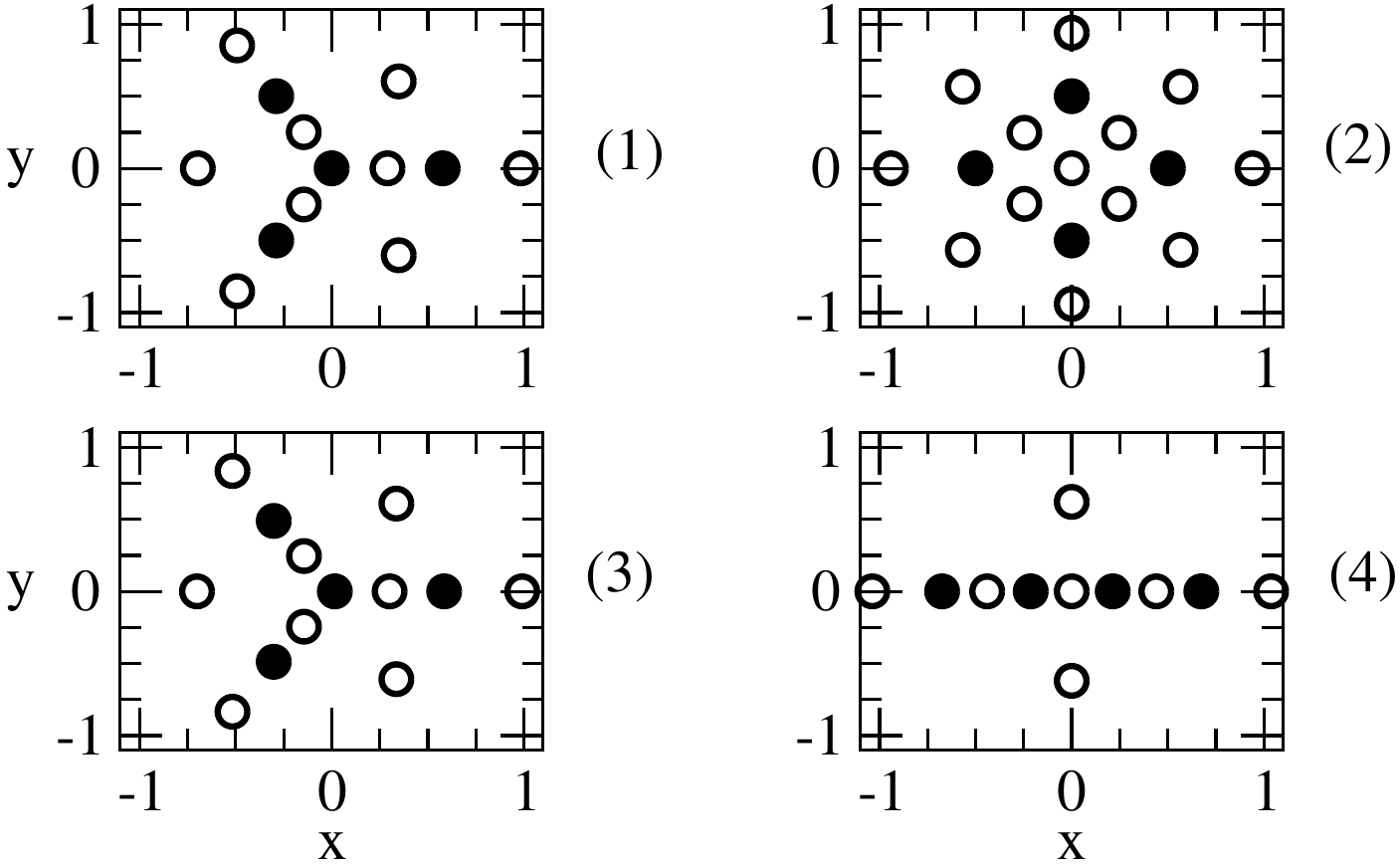}

\caption{Central configurations $\mathbf{q}^{(k,l)}=(\mathbf{q}^{(k)T},\mathbf{q}_{n+1}^{(k,l)T})^{T}$
in the case $n=4$. The solutions $\mathbf{q}^{(k)}$, corresponding
to the solutions $\mathbf{q}_{0}^{(k)}$ of the $n$-body problem,
are marked with filled circles and are shown in each of the four plots.
For each $\mathbf{q}^{(k)}$, the solutions $\mathbf{q}_{n+1}^{(k,l)}$,
corresponding to the restricted $(n+1)$-body problem, are marked
with open circles. The number of configurations for the $n$-body
problem is $N_{\textrm{sol}}(n)=4$, while the number of configurations
for the $(n+1)$-body problem is $N_{\textrm{sol}}(n+1)=38$. \label{fig:CC4}}
\end{figure}

\begin{figure}
\includegraphics[scale=0.5]{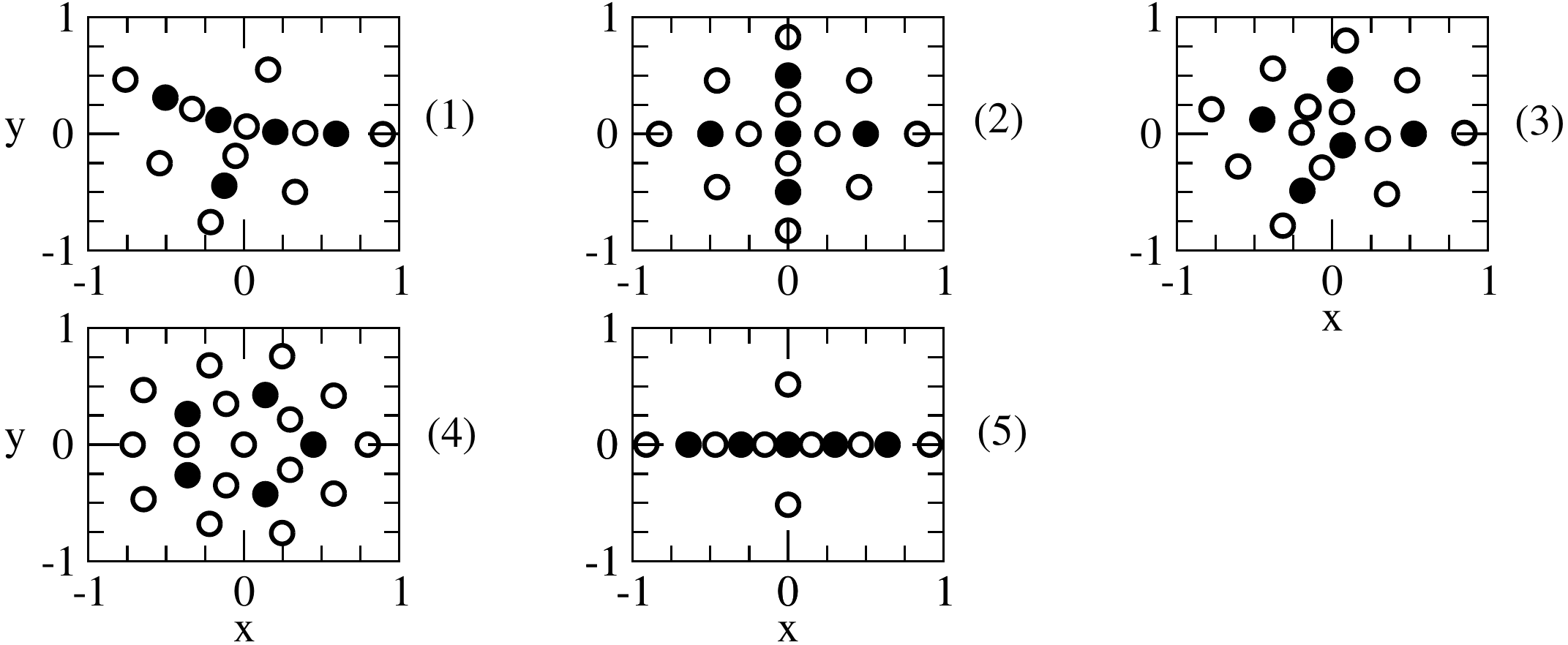}

\caption{The same as in Fig. \ref{fig:CC4} but for $n=5$. The number of configurations
for the $n$-body problem is $N_{\textrm{sol}}(n)=5$, while the corresponding
number of configurations for the $(n+1)$-body problem is $N_{\textrm{sol}}(n+1)=60$.\label{fig:CC5}}
\end{figure}

\begin{figure}
\includegraphics[scale=0.5]{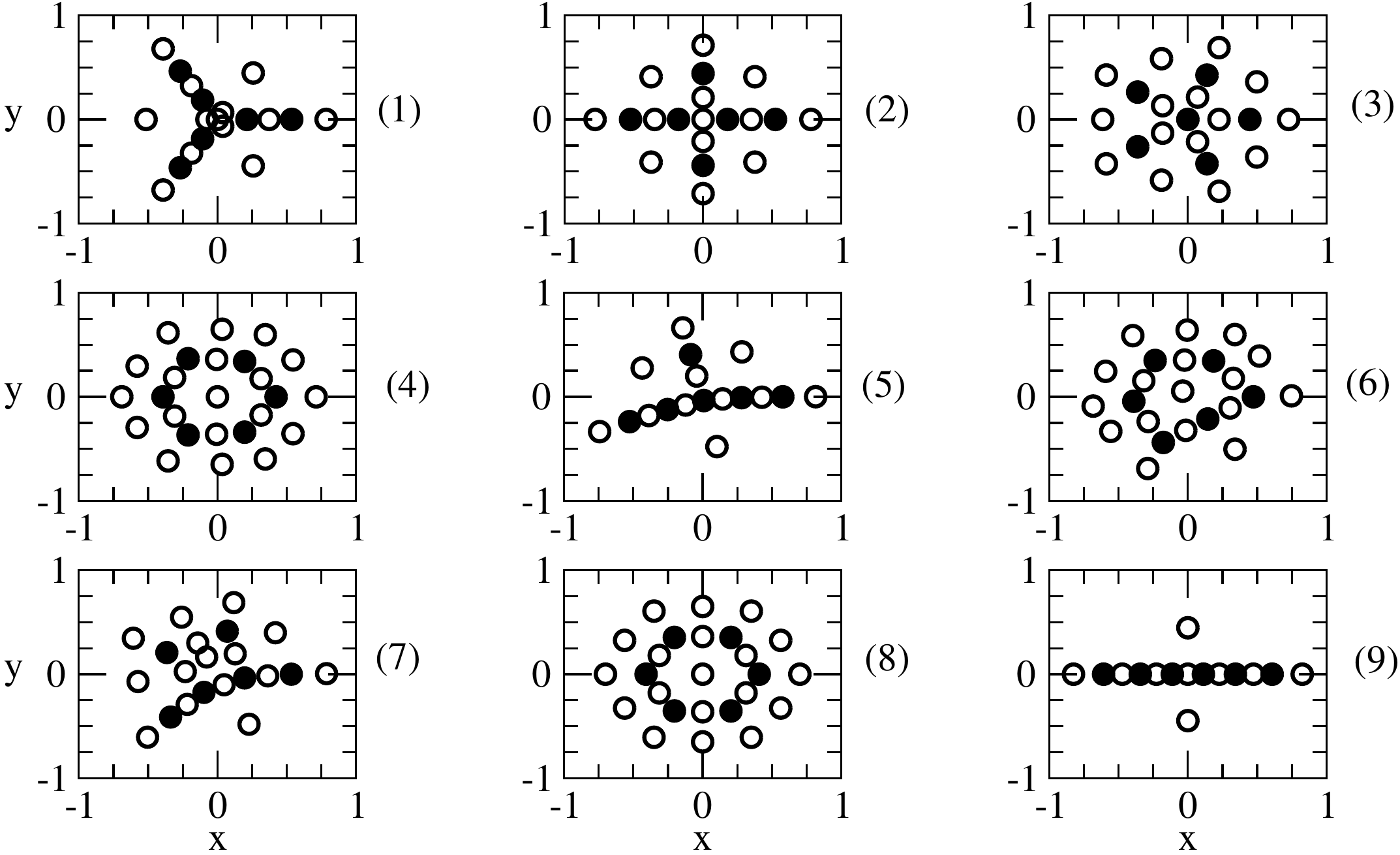}

\caption{The same as in Fig. \ref{fig:CC4} but for $n=6$. The number of configurations
for the $n$-body problem is $N_{\textrm{sol}}(n)=9$, while the number
of configurations for the $(n+1)$-body problem is $N_{\textrm{sol}}(n+1)=131$.\label{fig:CC6}}
\end{figure}

\begin{figure}
\includegraphics[scale=0.5]{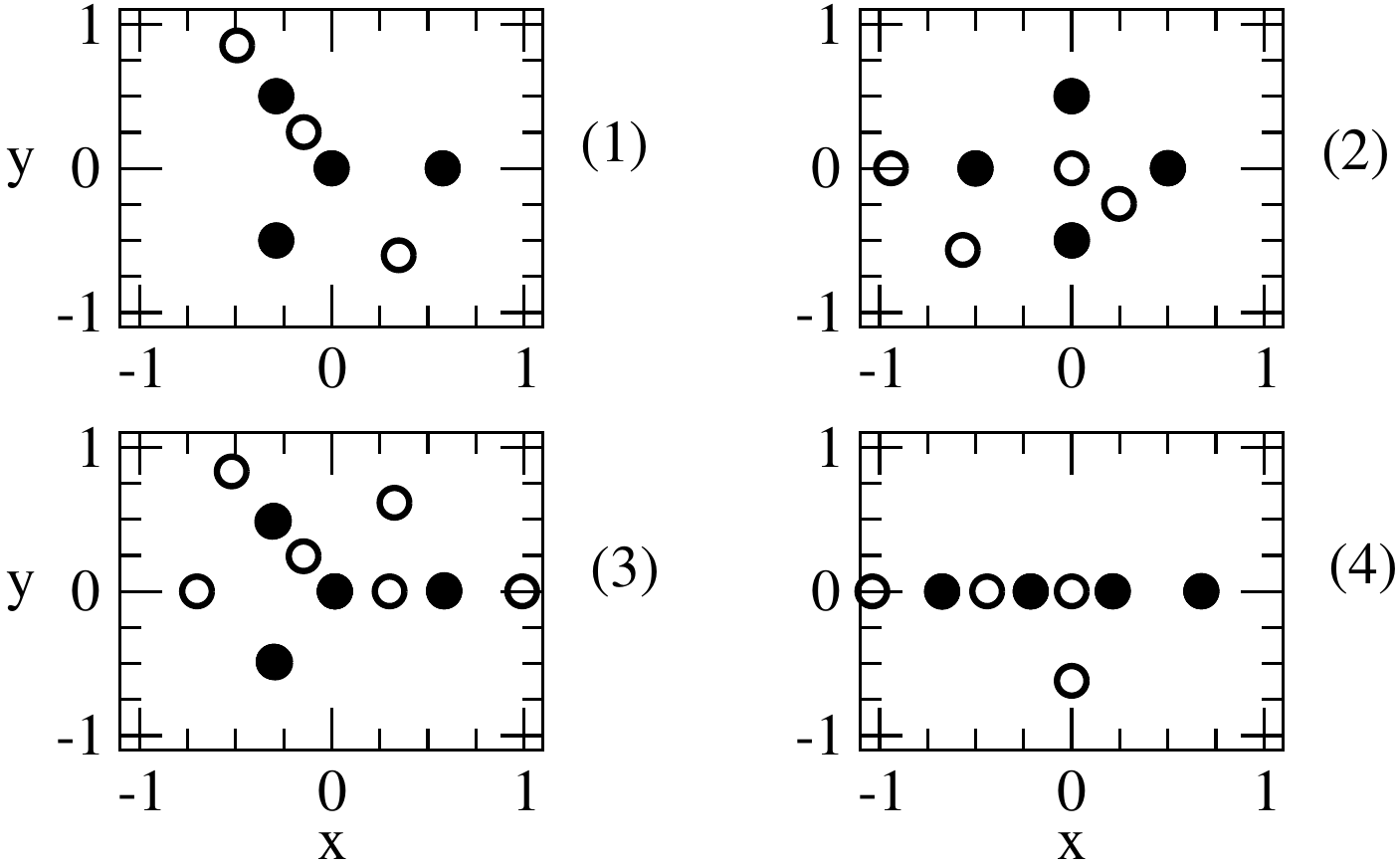}

\caption{Distinct central configurations (excluding discrete rotated and reflected
solutions) in the case $n=4$. The number of configurations for the
$n$-body problem is $N_{\textrm{sol}}(n)=4$, while the number of
distinct configurations for the $(n+1)$-body problem is $N_{\textrm{sol}}^{0}(n+1)=17$.
Note that these configurations also correspond to the direct method.
\label{fig:CC4a} }
\end{figure}

\begin{figure}
\includegraphics[scale=0.5]{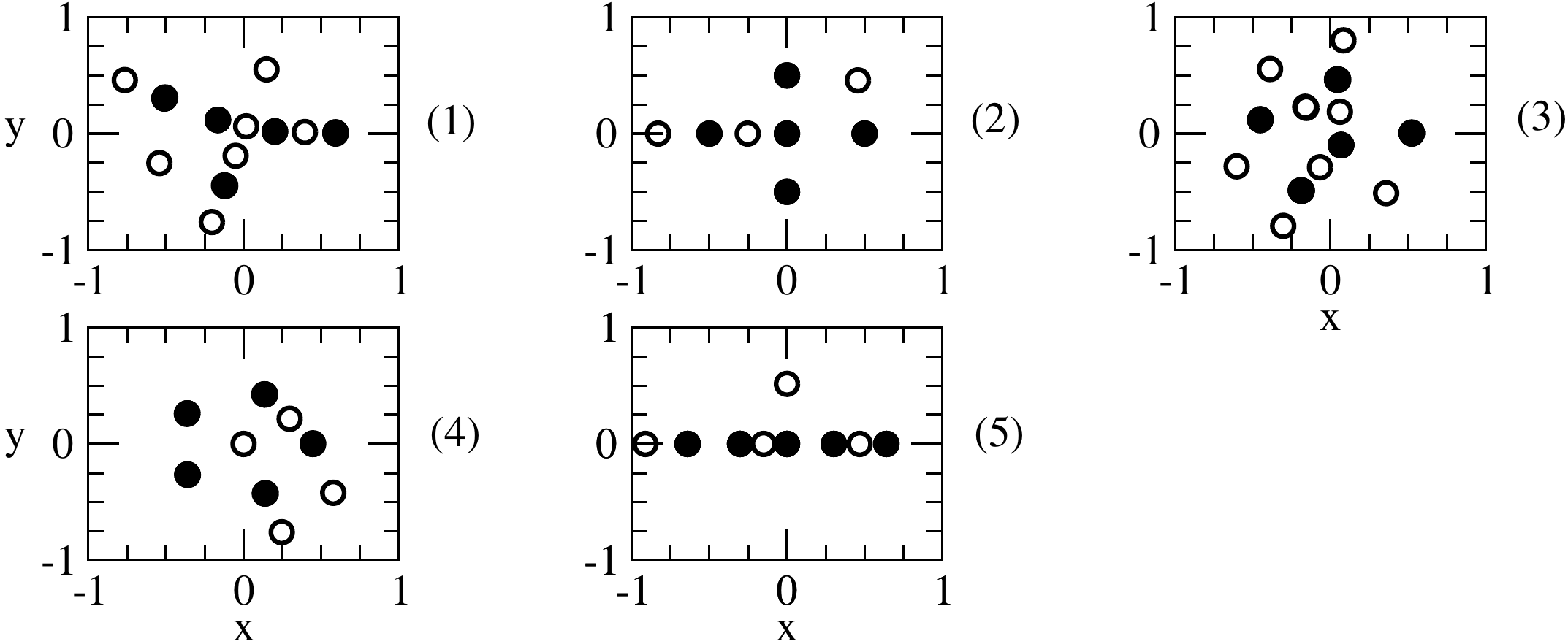}

\caption{The same as in Fig. \ref{fig:CC4a} but for $n=5$. The number of
configurations for the $n$-body problem is $N_{\textrm{sol}}(n)=5$,
while the number of distinct configurations for the $(n+1)$-body
problem is $N_{\textrm{sol}}^{0}(n+1)=27$. \label{fig:CC5a}}
\end{figure}

\begin{figure}
\includegraphics[scale=0.5]{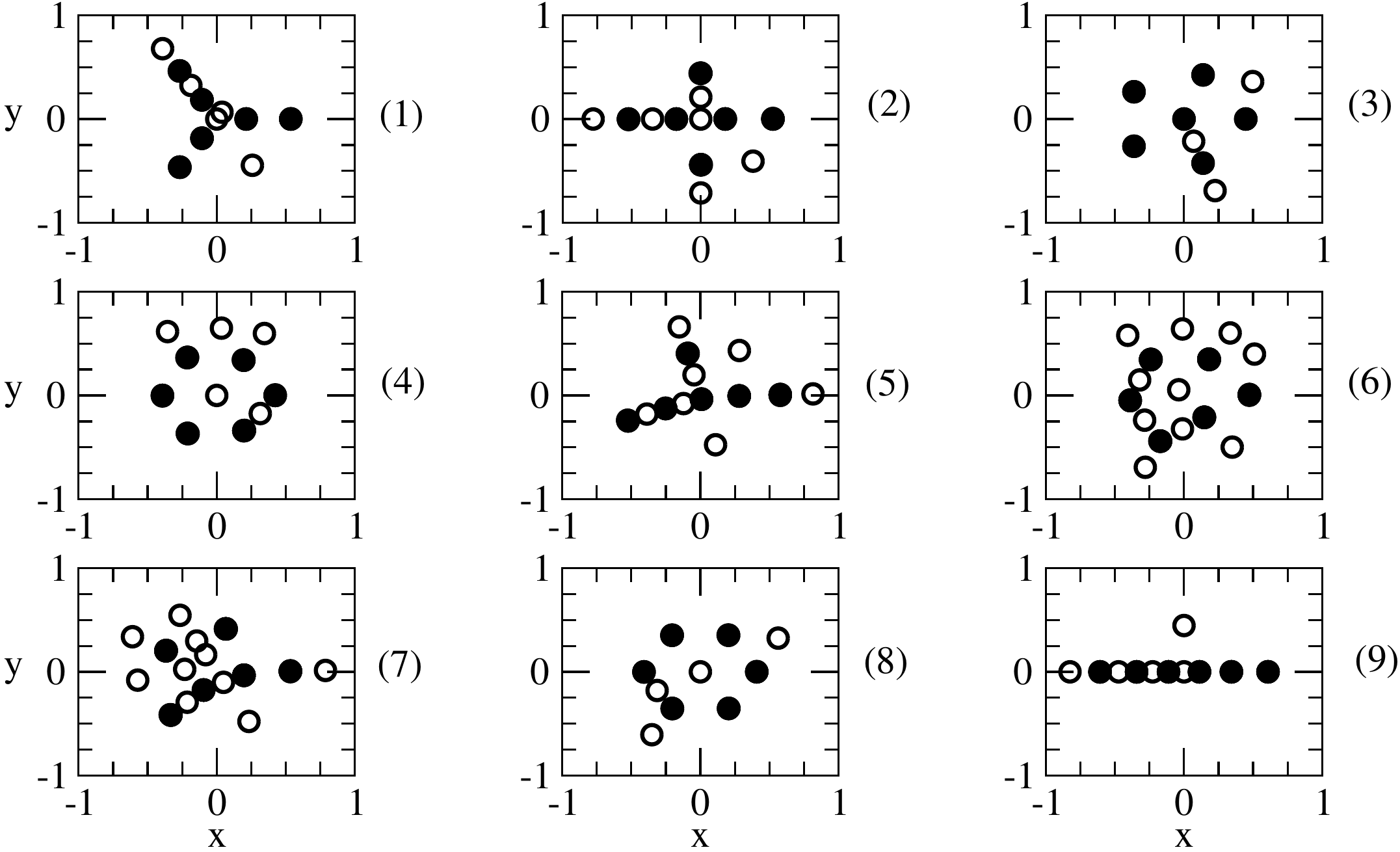}

\caption{The same as in Fig. \ref{fig:CC4a} but for $n=6$. The number of
configurations for the $n$-body problem is $N_{\textrm{sol}}(n)=9$,
while the number of distinct configurations for the $(n+1)$-body
problem is $N_{\textrm{sol}}^{0}(n+1)=55$. \label{fig:CC6a}}
\end{figure}

\begin{table}
\caption{Numbers of central configurations, the RMS errors, and the computational
times corresponding to the analytic-continuation method and the direct
method. Here, $n$ is the number of point masses, $\varepsilon$ the
mass parameter, $N_{\textrm{sol}}(n)$ and $N_{\textrm{sol}}(n+1)$
the number of central configurations computed by the analytic-continuation
method for the $n$- and $(n+1)$-body problems, respectively, $N_{\textrm{sol}}^{0}(n+1)$
the number of distinct solutions (excluding discrete rotated and reflected
solutions), $\widehat{N}_{\textrm{sol}}(n+1)$ the number of central
configurations computed by the direct method for the $(n+1)$-body
problem, $\Delta q_{0}$ the average RMS of the absolute errors in
Cartesian coordinates, and $\Delta R$ the average RMS of the absolute
errors in mutual distances. The RMS value $x.yz(e)$ should be understand
as $x.yz\times10^{e}$, and the computational time is given in minutes:seconds.
\label{tab:RMS} }

\medskip{}

\resizebox{\textwidth}{!}{

\begin{tabular}{ccccccccccc}
\hline 
\multirow{2}{*}{$n$} & \multirow{2}{*}{$\varepsilon$} & \multicolumn{5}{c}{Analytic-continuation method} & \multirow{2}{*}{} & \multicolumn{3}{c}{Direct method}\tabularnewline
\cline{3-7} \cline{4-7} \cline{5-7} \cline{6-7} \cline{7-7} \cline{9-11} \cline{10-11} \cline{11-11} 
 &  & $N_{\textrm{sol}}(n)$ & $N_{\textrm{sol}}(n+1)$ & $N_{\textrm{sol}}^{0}(n+1)$ & $\begin{array}{c}
\textrm{RMS}\\
\Delta q_{0}
\end{array}$ & Time &  & $\widehat{N}_{\textrm{sol}}(n+1)$ & $\begin{array}{c}
\textrm{RMS}\\
\Delta R
\end{array}$ & Time\tabularnewline
\cline{1-7} \cline{2-7} \cline{3-7} \cline{4-7} \cline{5-7} \cline{6-7} \cline{7-7} \cline{9-11} \cline{10-11} \cline{11-11} 
\multirow{3}{*}{4} & $10^{-8}$ & \multirow{3}{*}{4} & \multirow{3}{*}{38} & \multirow{3}{*}{17} & 3.03(-8) & 0:30 &  & \multirow{3}{*}{17} & 6.85(-8) & 8:13\tabularnewline
 & $10^{-9}$ &  &  &  & 3.04(-9) & 0:22 &  &  & 5.86(-9) & 6:14\tabularnewline
 & $10^{-10}$ &  &  &  & 4.44(-10) & 0:20 &  &  & 5.76(-10) & 5:32\tabularnewline
\cline{1-7} \cline{2-7} \cline{3-7} \cline{4-7} \cline{5-7} \cline{6-7} \cline{7-7} \cline{9-11} \cline{10-11} \cline{11-11} 
\multirow{3}{*}{5} & $10^{-8}$ & \multirow{3}{*}{5} & \multirow{3}{*}{60} & \multirow{3}{*}{27} & 1.09(-8) & 0:54 &  & \multirow{3}{*}{27} & 3.24(-9) & 8:20\tabularnewline
 & $10^{-9}$ &  &  &  & 1.10(-9) & 0:53 &  &  & 3.41(-10) & 8:31\tabularnewline
 & $10^{-10}$ &  &  &  & 2.40(-10) & 0:53 &  &  & 1.03(-10) & 10:18\tabularnewline
\cline{1-7} \cline{2-7} \cline{3-7} \cline{4-7} \cline{5-7} \cline{6-7} \cline{7-7} \cline{9-11} \cline{10-11} \cline{11-11} 
\multirow{3}{*}{6} & $10^{-8}$ & \multirow{3}{*}{9} & \multirow{3}{*}{131} & \multirow{3}{*}{55} & 4.64(-8) & 2:54 &  & \multirow{3}{*}{55} & 1.19(-9) & 17:05\tabularnewline
 & $10^{-9}$ &  &  &  & 4.64(-9) & 2:57 &  &  & 1.43(-10) & 15:42\tabularnewline
 & $10^{-10}$ &  &  &  & 4.69(-10) & 2:52 &  &  & 2.81(-10) & 14:58\tabularnewline
\hline 
\end{tabular}

}
\end{table}

\subsection{Balanced configurations}

The balanced configurations computed by the analytic-continuation
method in the cases $n=4$, $5$, and $6$ are illustrated in Figs.
\ref{fig:BC4}, \ref{fig:BC5}, and \ref{fig:bc6}, respectively,
while the numbers of balanced configurations, the RMS errors, and
the computational times for the analytic-continuation and the direct
method are given in Table \ref{tab:RMS-1}. As in the case of central
configurations, we are led to the following conclusions.
\begin{enumerate}
\item If a configuration $\mathbf{q}_{0}^{(k)}$ of the $n$- body problem
has an axis of symmetry with respect to the $x$- or $y$-axis, or
is symmetric with respect to the origin of the coordinate system,
the same happens with the corresponding configuration $\mathbf{q}^{(k,l)}$
of the $(n+1)$-body. Parenthetically we note that the number of solutions
is significantly higher than that of central configurations.
\item For the mass parameter $\varepsilon=10^{-8},10^{-9},10^{-10}$, (i)
the number of distinct solutions (excluding reflected solutions) computed
by the analytic-continuation method coincides with the number of solutions
computed by the direct method, i.e., $N_{\textrm{sol}}^{0}(n+1)=\widehat{N}_{\mathrm{sol}}(n+1)$,
(ii) the RMS error $\Delta q_{0}$ decreases with $\varepsilon$,
(iii) the RMS error $\Delta R$ is small, and (iv) the analytic-continuation
method is much more efficient than the direct method.
\end{enumerate}
\begin{figure}
\includegraphics[scale=0.5]{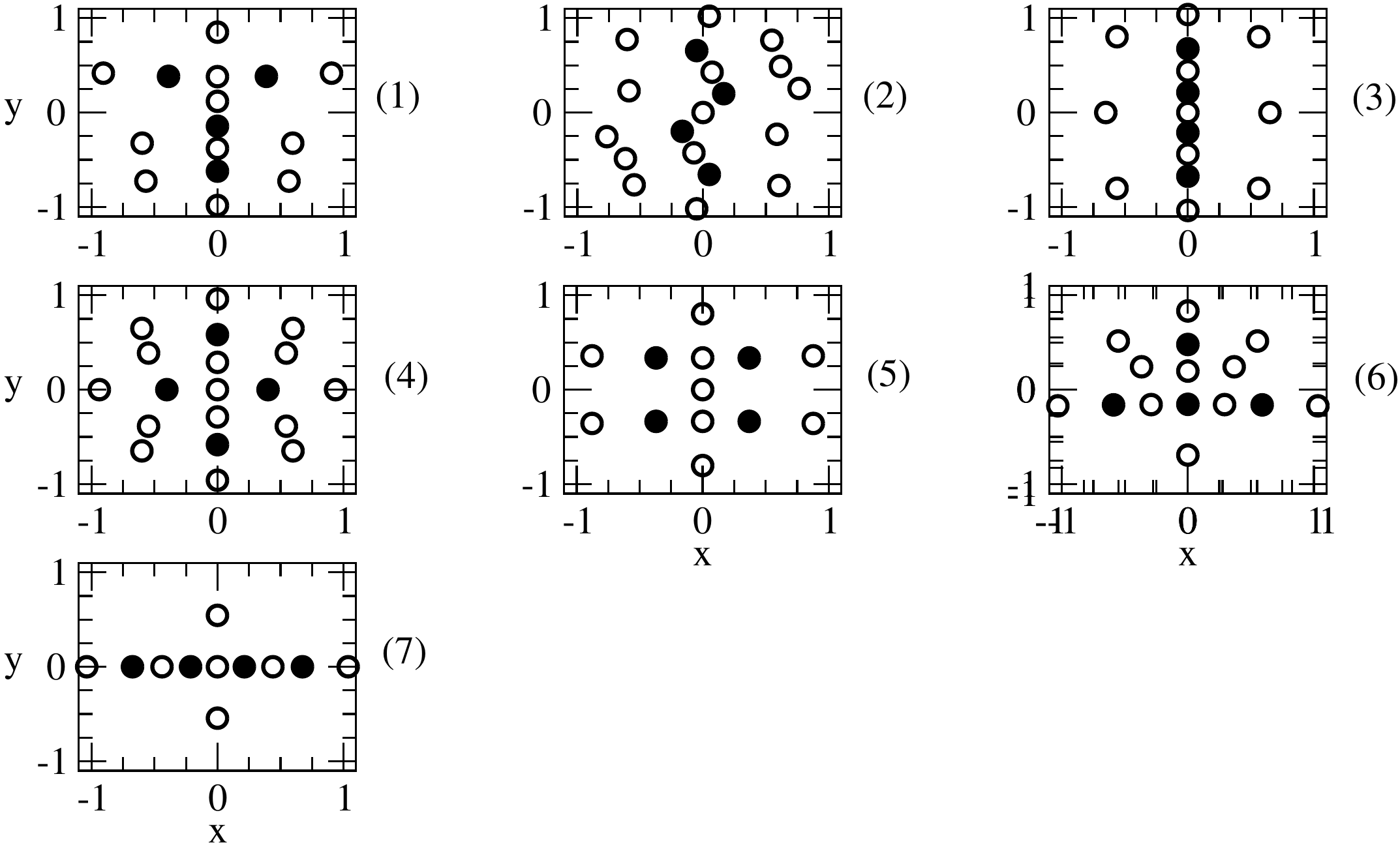}

\caption{Balanced configurations $\mathbf{q}^{(k,l)}=(\mathbf{q}^{(k)T},\mathbf{q}_{n+1}^{(k,l)T})^{T}$
in the case $n=4$. The solutions $\mathbf{q}^{(k)}$, corresponding
to the solutions $\mathbf{q}_{0}^{(k)}$ of the $n$-body problem,
are marked with filled circles, and for each $\mathbf{q}^{(k)}$,
the solutions $\mathbf{q}_{n+1}^{(k,l)}$, corresponding to the restricted
$(n+1)$-body problem, are marked with open circles. The number of
configurations for the $n$-body problem is $N_{\textrm{sol}}(n)=7$,
while the number of configurations for the $(n+1)$-body problem is
$N_{\textrm{sol}}(n+1)=79$. \label{fig:BC4}}
\end{figure}

\begin{figure}
\includegraphics[scale=0.5]{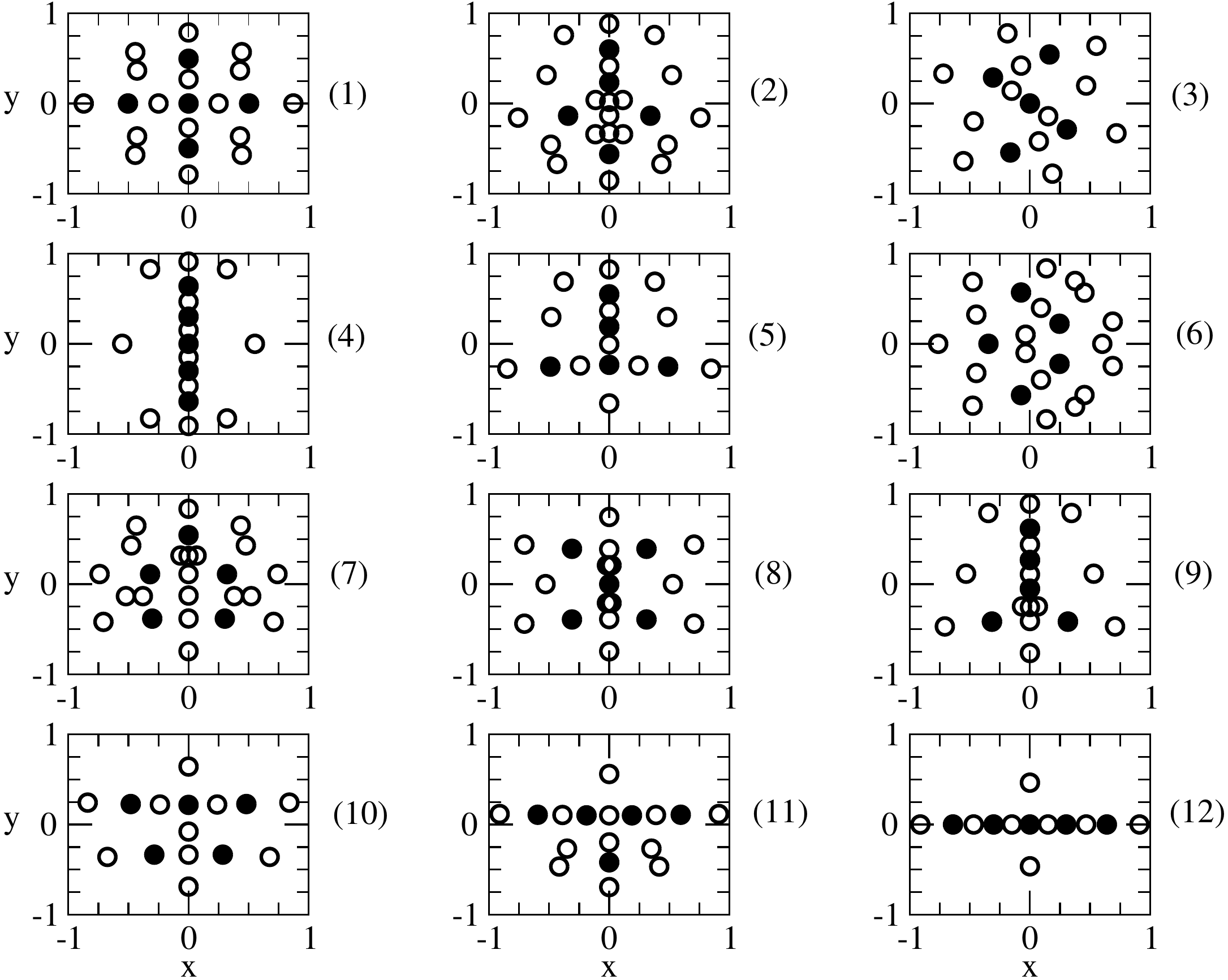}

\caption{The same as in Fig. \ref{fig:BC4} but for $n=5$. The number of configurations
for the $n$-body problem is $N_{\textrm{sol}}(n)=12$, while the
number of configurations for the $(n+1)$-body problem is $N_{\textrm{sol}}(n+1)=170$.\label{fig:BC5}}
\end{figure}

\begin{figure}
\includegraphics[scale=0.5]{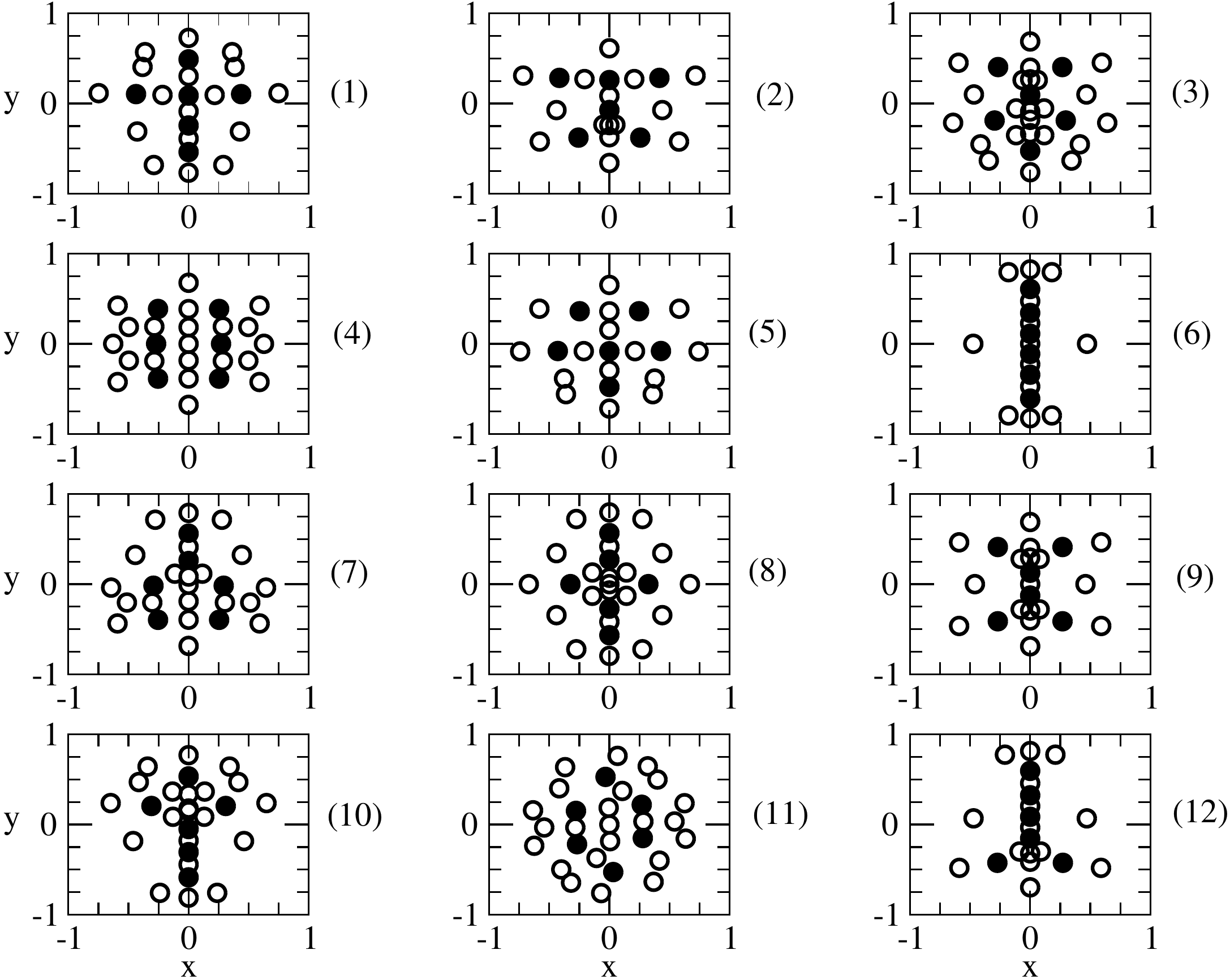}

\caption{The same as in Fig. \ref{fig:BC4} but for $n=6$. The number of configurations
for the $n$-body problem is $N_{\textrm{sol}}(n)=22$, while the
number of configurations for the $(n+1)$-body problem is $N_{\textrm{sol}}(n+1)=366$.
\label{fig:bc6}}
\end{figure}

\newpage{}

\begin{center}
\begin{figure}
\includegraphics[scale=0.167]{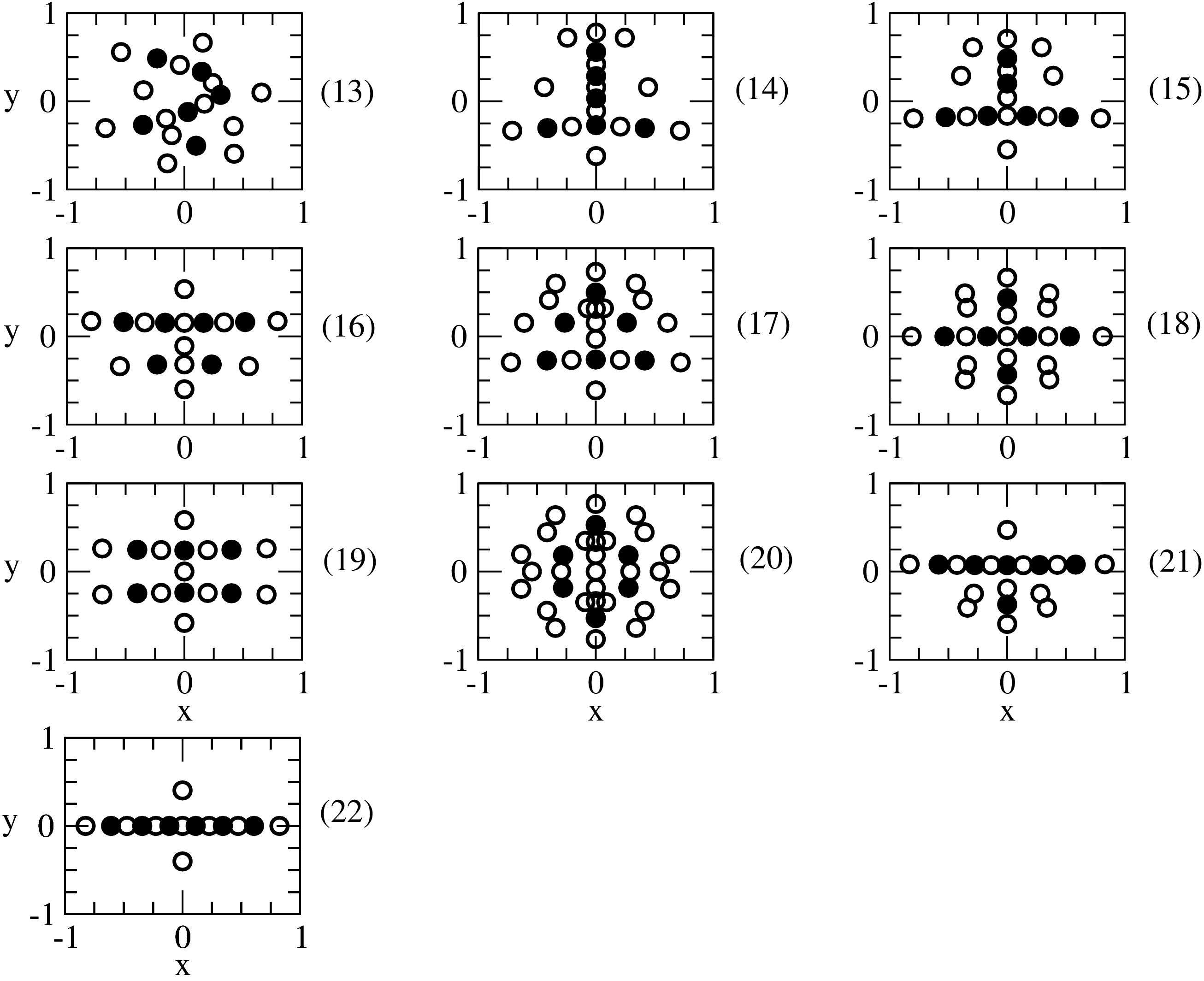}

\caption{Continuation of Fig. \ref{fig:bc6}. \label{fig:bc6-1}}
\end{figure}
\end{center}




\begin{table}
\caption{The same as in Table \ref{tab:RMS} but for balanced configurations.
Here, the number of distinct solutions computed by the analytic-continuation
method $N_{\textrm{sol}}^{0}(n+1)$ excludes reflected solutions.
\label{tab:RMS-1} }

\medskip{}

\resizebox{\textwidth}{!}{

\begin{tabular}{ccccccccccc}
\hline 
\multirow{2}{*}{$n$} & \multirow{2}{*}{$\varepsilon$} & \multicolumn{5}{c}{Analytic-continuation method} & \multirow{2}{*}{} & \multicolumn{3}{c}{Direct method}\tabularnewline
\cline{3-7} \cline{4-7} \cline{5-7} \cline{6-7} \cline{7-7} \cline{9-11} \cline{10-11} \cline{11-11} 
 &  & $N_{\textrm{sol}}(n)$ & $N_{\textrm{sol}}(n+1)$ & $N_{\textrm{sol}}^{0}(n+1)$ & $\begin{array}{c}
\textrm{RMS}\\
\Delta q_{0}
\end{array}$ & Time &  & $\widehat{N}_{\textrm{sol}}(n+1)$ & $\begin{array}{c}
\textrm{RMS}\\
\Delta R
\end{array}$ & Time\tabularnewline
\cline{1-7} \cline{2-7} \cline{3-7} \cline{4-7} \cline{5-7} \cline{6-7} \cline{7-7} \cline{9-11} \cline{10-11} \cline{11-11} 
\multirow{3}{*}{4} & $10^{-8}$ & \multirow{3}{*}{7} & \multirow{3}{*}{79} & \multirow{3}{*}{42} & 1.84(-8) & 0:31 &  & \multirow{3}{*}{42} & 4.54(-9) & 5:34\tabularnewline
 & $10^{-9}$ &  &  &  & 2.09(-9) & 0:31 &  &  & 1.80(-9) & 5:47\tabularnewline
 & $10^{-10}$ &  &  &  & 5.81(-10) & 0:30 &  &  & 4.18(-10) & 6:04\tabularnewline
\cline{1-7} \cline{2-7} \cline{3-7} \cline{4-7} \cline{5-7} \cline{6-7} \cline{7-7} \cline{9-11} \cline{10-11} \cline{11-11} 
\multirow{3}{*}{5} & $10^{-8}$ & \multirow{3}{*}{12} & \multirow{3}{*}{170} & \multirow{3}{*}{96} & 1.22(-8) & 1:10 &  & \multirow{3}{*}{96} & 6.35(-9) & 9:17\tabularnewline
 & $10^{-9}$ &  &  &  & 1.29(-9) & 1:10 &  &  & 8.89(-10) & 9:10\tabularnewline
 & $10^{-10}$ &  &  &  & 3.02(-10) & 1:12 &  &  & 4.30(-10) & 9:24\tabularnewline
\cline{1-7} \cline{2-7} \cline{3-7} \cline{4-7} \cline{5-7} \cline{6-7} \cline{7-7} \cline{9-11} \cline{10-11} \cline{11-11} 
\multirow{3}{*}{6} & $10^{-8}$ & \multirow{3}{*}{22} & \multirow{3}{*}{366} & \multirow{3}{*}{210} & 1.04(-8) & 3:01 &  & \multirow{3}{*}{210} & 1.19(-8) & 18:13\tabularnewline
 & $10^{-9}$ &  &  &  & 1.07(-9) & 2.19 &  &  & 1.44(-9) & 18:42\tabularnewline
 & $10^{-10}$ &  &  &  & 1.52(-10) & 2:28 &  &  & 5.67(-10) & 18:50\tabularnewline
\hline 
\end{tabular}

}
\end{table}

\section{Conclusions}

Planar central and balanced configurations in the $(n+1)$-body problem
with a small mass have been analyzed from a numerical point of view.
For this purpose, two algorithms have been designed. The first one
relies on a direct solution method of the $(n+1)$-body problem by
using a stochastic optimization approach, while the second one is
based on an analytic-continuation method. The analytic-continuation
algorithm involves three computational steps. These include the solutions
of the $n$-body and the restricted $(n+1)$-body problem, and the
application of a local search procedure to compute the final $(n+1)$-body
configuration in the neighborhood of the configuration obtained at
the first two steps. Our numerical experiments have showed that 
\begin{enumerate}
\item if a configuration of the $n$-body problem has an axis of symmetry,
the corresponding configuration of the $(n+1)$-body problem inherits
this symmetry,
\item for sufficiently small values of the mass $m_{n+1}$, both algorithms
deliver almost identical configurations,
\item the algorithm based on the analytic-continuation method is on average
7 times faster than the algorithm based on the direct method.
\end{enumerate}
\begin{acknowledgement*}
Alexandru Doicu and Lei Zhao were supported by DFG ZH 605/1-1.
\end{acknowledgement*}

\end{document}